\documentclass[a4paper,11pt]{article}

\usepackage{amsfonts}
\usepackage{amssymb}
\usepackage{amsthm}
\usepackage{amsmath}
\usepackage{graphicx}
\usepackage{empheq}
\usepackage{indentfirst}
\usepackage{mathrsfs}
\usepackage{graphicx}

\usepackage{caption}
\usepackage{subcaption}

\newtheorem{theorem}{\textbf{Theorem}}[section]
\newtheorem{lemma}{\textbf{Lemma}}[section]
\newtheorem{proposition}{\textbf{Proposition}}[section]
\newtheorem{corollary}{\textbf{Corollary}}[section]
\newtheorem{remark}{\textbf{Remark}}[section]
\newtheorem{definition}{\textbf{Definition}}[section]

\allowdisplaybreaks[4]

\def\be{\begin{equation}}
\def\ee{\end{equation}}
\def\bea{\begin{eqnarray}}
\def\eea{\end{eqnarray}}
\def\bt{\begin{theorem}}
\def\et{\end{theorem}}
\def\bl{\begin{lemma}}
\def\el{\end{lemma}}
\def\br{\begin{remark}}
\def\er{\end{remark}}
\def\bp{\begin{proposition}}
\def\ep{\end{proposition}}
\def\bc{\begin{corollary}}
\def\ec{\end{corollary}}
\def\bd{\begin{definition}}
\def\ed{\end{definition}}

\def\bfu{\mathbf{u}}
\def\baru{\overline{\mathbf{u}}}
\def\tu{\widetilde{\mathbf{u}}}
\def\tphi{\widetilde{\phi}}
\def\12{\frac{1}{2}}
\def\bfH{\mathbf{H}}
\def\tw{\widetilde{W}e^{-1}}

\def\sw{{We}^\ast}
\def\L{_{L^2}}
\def\H{_{H^1}}

\title{A second order in time, uniquely solvable, unconditionally stable numerical scheme for Cahn-Hilliard-Navier-Stokes equation}

\author{
 {\sc Daozhi Han}\footnote{Department of Mathematics, Florida State University, Tallahassee, FL 32306, USA.
Email: \emph{dhan@math.fsu.edu}} ,\ \
 {\sc Xiaoming Wang}\footnote{ Department of Mathematics, Florida
State University, Tallahassee, FL 32306, USA. Email:
\emph{wxm@math.fsu.edu}}
 }

\date{\today}

\begin{document}
\maketitle

\begin{abstract}
We propose a novel second order in time numerical scheme for Cahn-Hilliard-Navier-Stokes phase field model with matched density. The scheme is based on second order convex-splitting for the Cahn-Hilliard equation and pressure-projection for the Navier-Stokes equation. We show that the  scheme is mass-conservative, satisfies a modified energy law and is therefore unconditionally stable. Moreover, we prove that the scheme is unconditionally uniquely solvable at each time step by exploring the monotonicity associated with the scheme.  Thanks to the weak coupling of the scheme, we design an efficient Picard iteration procedure to further decouple the computation of Cahn-Hilliard equation and Navier-Stokes equation. We implement the scheme by the mixed finite element  method. Ample numerical experiments are performed to validate the accuracy and efficiency of the numerical scheme.
\end{abstract}


\section{Introduction}
In this work, we are interested in solving numerically the Cahn-Hilliard-Navier-Stokes (CHNS) phase field model that describes the interface dynamics of a binary incompressible and macroscopically immiscible Newtonian fluids with matched density  and viscosity in a bounded domain $\Omega \subseteq \mathbb{R}^d, d=2, 3$. The non-dimensional system takes the explicit form as, cf. \cite{KKL2004}
\begin{align}
&\phi_t+ \nabla \cdot (\phi \mathbf{u})=\nabla \cdot(M(\phi) \nabla \mu), \quad  \text{ in } \Omega_T \label{CH} \\
&\mu=f_0^\prime(\phi)-\epsilon^2 \Delta \phi,  \quad \text{ in } \Omega_T \label{CP} \\
&\mathbf{u}_t-\frac{1}{Re} \Delta \mathbf{u}+\mathbf{u}\cdot \nabla \mathbf{u}+\nabla p=-\frac{\epsilon^{-1}}{\sw} \phi \nabla \mu,  \quad  \text{in } \Omega_T  \label{NS}\\
&\nabla \cdot \mathbf{u}=0, \quad  \text{in } \Omega_T \label{div} 
\end{align}
where $\mathbf{u}$ is the velocity field, $p$ is a modified pressure, $\phi $ is the phase field variable (order parameter), $\mu$ the chemical potential, $f_0(\phi)$ is the quartic homogeneous free energy density function
$
f_0(\phi)=\frac{1}{4}(1-\phi^2)^2,
$
and $\Omega_T:=\Omega \times (0,T)$ with $T>0$ a fixed constant.
$Re$ is the Reynolds number;  $\sw$ is the modified Weber number that measures the relative strengths of the kinetic and surface energies \cite{KKL2004}; $\epsilon$ is a dimensionless parameter that measures capillary width of the diffuse interface; $M(\phi)$ is the mobility function that incorporates the diffusional Peclet number $Pe$. We refer to \cite{KAD2007, LoTr1998} for the detailed non-dimensionalization of the CHNS system. 

We close the system with the following initial and boundary conditions
\begin{align}\label{BCs}
 \mathbf{u}&=0, \quad \text{on } \partial \Omega \times (0,T) \\
 \nabla \phi \cdot \mathbf{n}=\nabla \mu \cdot \mathbf{n} &=0, \quad \text{on } \partial \Omega \times (0,T) \\
 (\mathbf{u}, \phi)|_{t=0}&=(\mathbf{u}_0, \phi_0), \quad \text{in } \Omega.
\end{align}
Here $\mathbf{n}$ denotes the unit outer normal vector of the boundary $\partial \Omega$. 
It is clear that the CHNS system \eqref{CH}-\eqref{div} under the above boundary conditions is mass-conservative, 
\begin{align}\label{Mass-c}
\frac{d}{dt}\int_\Omega \phi\, dx=0,
\end{align}
and energy-dissipative 
\begin{align}\label{ConEnL}
\frac{d}{dt}E_{tot}(\mathbf{u}, \phi)=-\frac{1}{Re}\int_\Omega  |\nabla \mathbf{u}|^2\,dx -\frac{\epsilon^{-1}}{\sw} \int_\Omega M(\phi)|\nabla \mu|^2\, dx,
\end{align}
where the total energy $E_{tot}$ is defined as
\begin{align}\label{Etot}
E_{tot}(\mathbf{u}, \phi)=\int_{\Omega}\frac{1}{2}|\mathbf{u}|^2\, dx+ \frac{1}{\sw}\int_{\Omega} \big( \frac{1}{\epsilon} f_0(\phi)+\frac{\epsilon}{2}|\nabla \phi|^2\big)\, dx.
\end{align}
The first term on the right hand side of equation \eqref{Etot} is the total kinetic energy, and the term, denoted by $E_{f}$ throughout, is a measure of the surface energy of the fluid system. 

The CHNS phase field model \eqref{CH}-\eqref{div} is proposed 
as an alternative of sharp interface model to describe the dynamics of two phase, incompressible, and macroscopically immiscible Newtonian fluids with matched density, cf.  \cite{HoHa1977, GPV1996, AMW1998, LoTr1998, LiSh2003}. 
In contrast to the sharp interface model, the diffuse interface model recognizes the micro-scale mixing and hence treats the interface of two fluids as a transition layer with small but non-zero width $\epsilon$. Although the region is thin, it may play an important role during topological transition like interface pinchoff or reconnection \cite{LoTr1998}. One then introduces an order parameter $\phi$, for instance the concentration difference, which takes the value $1$ in the bulk of one fluid and $-1$ in regions filled by the other fluid and varies continuously between $1$ and $-1$ over the interfacial region. One can view the zero level set of the order parameter as the averaged interface of the mixture. Thus, the dynamics of the interface can be simulated on a fixed grid without explicit interface tracking, which renders the diffuse interface method an attractive numerical approach for deforming interface problems. The CHNS diffuse interface model has been successfully employed for the simulations of two-phase flow in various contexts. We refer the readers to \cite{AMW1998, KKL2004} and references therein for its diverse applications.

In this work, we assume that $m_1 \leq M(\phi) \leq m_2$ for constants $0<m_1 \leq m_2$. We point out that the degenerate mobility function may be more physically relevant, as it guarantees the order parameter stays within the physical bound $\phi \in [-1, 1]$ \cite{Boyer1999}, though uniqueness of weak solutions is still open even for the Cahn-Hilliard equation. Recent numerical experiments \cite{ZhWa2010} also indicate that the Cahn-Hilliard equation with degenerate mobility may be more accurate  for immiscible binary fluids.  Numerical resolution of the degenerate case is a subtle matter and beyond the scope of our current work (cf. \cite{BBG1999, CeGa2013} for the case of Cahn-Hilliard equation).

There are several challenges in solving the system \eqref{CH}-\eqref{div} numerically. First of all, the small interfacial width $\epsilon$ introduces tremendous amount of stiffness into the system (large spatial derivative within the interfacial region). It demands the numerical scheme to be unconditionally stable so that the stiffness can be handled
with ease. The resulting numerical scheme tends to be nonlinear and therefore poses challenge in proving unconditionally unique solvability. A popular strategy in discretizing the Cahn-Hilliard equation (Eqs \eqref{CH}--\eqref{CP}) in time is based on the convex-splitting of the free energy functional $E_f$, i.e., treating the convex part of the functional implicitly and concave part explicitly, an idea dates back to Eyre \cite{Eyre1998}. The design of convex-splitting scheme yields not only unconditional stability, but also unconditionally unique solvability for systems with symmetric structures\cite{Wise2010, CSW2013}. However, the variational approach for proving unique solvability (see the references above) is not applicable to the CHNS system since the advection term in Navier-Stokes equation (Eq.\eqref{NS}) breaks the symmetry. In addition, the stiffness issue naturally requires adaptive mesh refinement in order to reduce the computational cost.  Secondly, when it comes to solving the Navier-Stokes equation, one always faces the  difficulty of the coupling between velocity and pressure. The common practice is to use the well-known Chorin-Temam type pressure projection scheme, see \cite{GMS2006} for a general review. Lastly, higher order scheme is always preferable from the accuracy point-of-view. Yet, it is a challenge to design higher order scheme for a nonlinear system while maintaining the unconditional stability.
 
There have been many works on the numerical resolution of the CHNS system, see a comprehensive summary by Shen \cite{Shen2012}. Here we survey several papers that are especially relevant to ours.  In \cite{KKL2004}, Kim, Kang and Lowengrub proposed a conservative, second-order accurate fully implicit discretization of the CHNS system. The update of the pressure in the Navier-Stokes equation is based on an approximate pressure projection method.  To ensure the unconditional stability, they introduce a non-linear stabilization term to the Navier-Stokes solver. The scheme is strongly coupled and highly nonlinear, for which they design a multigrid iterative solver. The authors point out (without proof) that a restriction on the time-step size may be needed for the unique solvability of the scheme. In \cite{Feng2006}, Feng analyses a first-order in time, fully discrete finite element approximation of the CHNS system. He shows that his scheme is unconditionally energy-stable and convergent, but gives no analysis on unique solvability. 
Kay, Styles and Welford \cite{KSW2008} also studied a first-order in time, finite element approximation of CHNS system. In contrast to Feng's scheme, the velocity in the Cahn-Hilliard equation \eqref{CH} is discretized explicitly at the discrete time level. Thus the computation of the Cahn-Hilliard equation is fully decoupled from that of Navier-Stokes equation. Moreover, the unique solvability of the overall scheme can be established easily by exploring the gradient flow structure of the Cahn-Hilliard equation. However, a CFL condition has to be imposed for the scheme to be stable. See \cite{Minjeaud2013} for an operator-splitting strategy in decoupling the computation of Cahn-Hilliard equation and Navier-Stokes equation which still preserves the unconditional stability (without decoupling the pressure and velocity). Dong and Shen \cite{DoSh2012} recently derived a fully decoupled linear time stepping scheme for the CHNS system with variable density, which involves only constant matrices for all flow variables. 
However, there is no stability analysis on their numerical scheme.


In this paper, we propose a novel second order in time numerical scheme for Cahn-Hilliard-Navier-Stokes phase field model with matched density. The scheme is based on second order convex-splitting for the Cahn-Hilliard equation and pressure-projection for the Navier-Stokes equation. This scheme satisfies a modified energy law which mimics the continuous version of the energy law \eqref{ConEnL}, and is therefore unconditionally stable.  Moreover, we prove that the scheme is unconditionally uniquely solvable at each time step by exploring the monotonicity associated with the scheme.  Thanks to the weak coupling of the scheme, we design an efficient Picard iteration procedure to further decouple the computation of Cahn-Hilliard equation and Navier-Stokes equation. We implement the scheme by the mixed finite element  method. Ample numerical experiments are performed to validate the accuracy and efficiency of the numerical scheme.
The possibility of such a scheme is alluded in Remark 5.5 \cite{Shen2012}. A similar scheme without pressure-correction for Cahn-Hilliard-Brinkman equation is proposed in the concluding remarks of \cite{CSW2013}.

The rest of the paper is organized as follows. In section 2, we give the discrete time, continuous space scheme. We prove the mass-conservation, unconditional stability and unconditionally unique solvability in section 3. In section 4, the scheme is further discretized in space by mixed finite element approximation. An efficient Picard iteration procedure is proposed to solve the fully discrete equations. Finally, We provide some numerical experiments in section 5 to validate our numerical scheme.

\section{A Discrete Time, Continuous Space Scheme}
Let $\delta t >0$ be a time step size and set $t^k=k\delta t$ for $0\leq k \leq K=[T/\delta t]$.  Without ambiguity, we denote by $(f,g)$ the $L^2$ inner product between functions $f$ and $g$. Also for convenience, the following notations will be used throughout this paper
\begin{subequations}\label{notas}
\begin{align}
&\phi^{k+\frac{1}{2}}=\frac{1}{2}(\phi^{k+1}+\phi^k), \quad \tphi^{k+\frac{1}{2}}=\frac{3\phi^k-\phi^{k-1}}{2},  \\
&\baru^{k+\frac{1}{2}}=\frac{\baru^{k+1}+\mathbf{u}^k}{2}, \quad \tu^{k+\frac{1}{2}}=\frac{3\mathbf{u}^k-\mathbf{u}^{k-1}}{2}.
\end{align}
\end{subequations}
We propose the semi-implicit, semi-discrete scheme in strong form as follows: 
\begin{align}
&\frac{\phi^{k+1}-\phi^k}{\delta t}=\nabla \cdot \big(M(\tphi^{k+\frac{1}{2}})\nabla \mu^{k+\frac{1}{2}}-\tphi^{k+\frac{1}{2}}\baru^{k+\frac{1}{2}}\big),\label{2ndCHNSa}\\
&\mu^{k+\frac{1}{2}}=\frac{1}{2}\big((\phi^{k+1})^2+(\phi^k)^2\big)\phi^{k+\frac{1}{2}}-\tphi^{k+\frac{1}{2}}-\epsilon^2 \Delta \phi^{k+\frac{1}{2}}, \label{2ndCHNSb}\\
&\frac{\baru^{k+1}-\bfu^k}{\delta t}-\frac{1}{Re}\Delta \baru^{k+\frac{1}{2}}+B(\tu^{k+\frac{1}{2}},  \baru^{k+\frac{1}{2}}) =-\nabla p^k-\frac{\epsilon^{-1}}{\sw}\tphi^{k+\frac{1}{2}} \nabla \mu^{k+\frac{1}{2}},\label{2ndCHNSc}\\
&\left\{
\begin{aligned}
&\frac{\bfu^{k+1}-\baru^{k+1}}{\delta t} + \frac{1}{2}\nabla(p^{k+1}-p^k)=0, \\
&\nabla \cdot \bfu^{k+1}=0,
\end{aligned} \label{2ndCHNSd}
\right.
\end{align}
with boundary conditions  
\begin{align}\label{2ndCHNSe}
&\nabla \phi^{k+1} \cdot  \mathbf{n}|_{\partial \Omega}=0, \quad \nabla \mu^{k+\frac{1}{2}}\cdot \mathbf{n}|_{\partial \Omega}=0,\quad \baru^{k+\frac{1}{2}}|_{\partial \Omega}=0, \quad \bfu^{k+1} \cdot \mathbf{n} |_{\partial \Omega}=0.
\end{align}
Here $B(\mathbf{u}, \mathbf{v}):=(\mathbf{u}\cdot \nabla)\mathbf{v}+\frac{1}{2}(\nabla \cdot \mathbf{u})\mathbf{v}$ is the skew-symmetric form of the nonlinear advection term in the Navier-Stokes equation \eqref{2ndCHNSc}, which is first introduced by Temam \cite{Temam1969}.  In the space continuous level,  $\nabla \cdot \tu^{k+\12}=0$, thus $B(\tu^{k+\frac{1}{2}},  \baru^{k+\frac{1}{2}})=\tu^{k+\12} \cdot \nabla \baru^{k+\12} $, which amounts to a second order semi-implicit discretization of the advection term. The skew symmetric form $B(\mathbf{u}, \mathbf{v})$ induces a trilinear form $b$ defined as, $\forall \mathbf{u}, \mathbf{v}, \mathbf{w} \in \mathbf{H}^1_0(\Omega)$ 
\begin{align}\label{trilinear}
b(\mathbf{u},\mathbf{v},\mathbf{w})=(B(\mathbf{u}, \mathbf{v}), \mathbf{w})=\frac{1}{2}\{(\mathbf{u}\cdot \nabla \mathbf{v}, \mathbf{w})-(\mathbf{u}\cdot \nabla \mathbf{w}, \mathbf{v})\}.
\end{align}
It follows immediately that $b(\mathbf{u},\mathbf{v},\mathbf{v})=0$ for any $\mathbf{u}, \mathbf{v} \in \mathbf{H}^1_0 (\Omega)$. This skew symmetry holds regardless of whether $\mathbf{u}, \mathbf{v}$ are divergence-free or not, which would help to preserve the stability when the scheme is further discretized in space.

The overall scheme \eqref{2ndCHNSa}--\eqref{2ndCHNSd} is based on the Crank-Nicolson time discretization and the second order Adams-Bashforth extrapolation. We note that the term $\frac{1}{2}\big((\phi^{k+1})^2+(\phi^k)^2\big)\phi^{k+\frac{1}{2}}-\tphi^{k+\frac{1}{2}}$ from the chemical potential equation \eqref{2ndCHNSb} is a second order approximation of the nonlinear term $f_0^\prime(\phi)$ (Eq.\eqref{CP}), which is derived according to a convex-splitting of the free energy density function $f_0(\phi)$. To see this, we rewrite $f_0(\phi)$ as the sum of a convex function and a concave function
$$
f_0(\phi)=f_v(\phi)+f_c(\phi):= \frac{1}{4}\phi^4+\big(-\frac{1}{2}\phi^2+\frac{1}{4}\big),
$$ 
and accordingly $f_0^\prime(\phi)=f_v^\prime(\phi)+f_c^\prime(\phi)$. The idea of convex-splitting is to use explicit discretization for the concave part (i.e. $f_c^\prime(\tphi^{k+\frac{1}{2}})$) and implicit discretization for the convex part. Thus we approximate $f_v^\prime(\phi^{k+\frac{1}{2}})$ by the Crank-Nicolson scheme
$$
f_v^\prime(\phi^{k+\frac{1}{2}})\approx \frac{f_v(\phi^{k+1})-f_v(\phi^k)}{\phi^{k+1}-\phi^k}=\frac{1}{2}[(\phi^{k+1})^2+(\phi^k)^2]\phi^{k+\frac{1}{2}}.
$$
Such a second order convex-splitting scheme is originally proposed and analysed in \cite{HWWL2009, BLWW2013} in the context of phase field crystal equation, see also \cite{SWWW2012} for applications in thin film epitaxy. We point out one can also approximate $f_0^\prime (\phi^{k+\frac{1}{2}})$ directly by Crank-Nicolson scheme \cite{KKL2004, GoHu2011} which would yield unconditional stability. The design of convex-splitting scheme enables us to prove not only unconditional stability but also unconditionally unique solvability of the overall scheme.

Eqs. \eqref{2ndCHNSc} and \eqref{2ndCHNSd} comprise  the second order incremental pressure projection method
of Van Kan type \cite{vanKan1986} with linear extrapolation for the nonlinear advection term.  The viscous step (Eq. \eqref{2ndCHNSc}) solves for an intermediate velocity $\baru^{k+1}$ (or, equivalently $\baru^{k+\frac{1}{2}}$) which is not divergence-free. The projection step (Eq. \eqref{2ndCHNSd}) is amount to \linebreak $\bfu^{k+1}=P_{\mathbf{H}} \baru^{k+1}$, where $P_{\mathbf{H}}$ is the Leray projection operator into $\mathbf{H}$:
\begin{align}
\mathbf{H}:=\{\mathbf{v} \in \mathbf{L}^2(\Omega); \nabla \cdot \mathbf{v}=0; \mathbf{v} \cdot \mathbf{n}|_{\partial \Omega}=0\}. \nonumber
\end{align}
The projection equation \eqref{2ndCHNSd} can also be solved in two sub-steps: first through a Pressure Poisson equation for the pressure increment 
\begin{equation} \label{PrePoi}
\left \{
\begin{aligned}
&\Delta (p^{k+1}-p^k)=\frac{2}{\delta t}\nabla \cdot \tu^{k+1}, \\
&\nabla (p^{k+1}-p^k) \cdot \mathbf{n}|_{\partial \Omega}=0.
\end{aligned}
\right.
\end{equation}
and then by an algebraic update for velocity
\begin{align} \label{AupVel}
&\bfu^{k+1}=\tu^{k+1}-\frac{\delta t}{2} \nabla (p^{k+1}-p^k). 
\end{align}
Variants of such a splitting method are analyzed in \cite{Shen1996} where it is shown (discrete time, continuous space) that the schemes are second order accurate for velocity in $l^2(0, T; L^2(\Omega))$ but only first order accurate for pressure in $l^\infty(0,T; L^2(\Omega))$. The loss of accuracy for pressure is due to the artificial boundary condition (cf. Eq. \eqref{PrePoi}) imposed on pressure \cite{ELi1995}.   We also remark that the Crank-Nicolson scheme with linear extrapolation is a popular time discretization for the Navier-Stokes equation. We refer to  \cite{Ingram2013} and references therein for analysis on this type of discretization.

Note that the projection step (Eq. \eqref{2ndCHNSd}) is decoupled from the rest of the equations. Moreover, the coupling between Eqs. \eqref{2ndCHNSa}-\eqref{2ndCHNSc} is fairly weak, thanks to the semi-implicit discretization. We see that the Cahn-Hilliard equation \eqref{2ndCHNSa} and \eqref{2ndCHNSb} is coupled with the Navier-Stokes equation \eqref{2ndCHNSa} only through the velocity $\baru^{k+\12}$ in the advection term of Eq. \eqref{2ndCHNSa} and the chemical potential $\mu^{k+\12}$ in the elastic forcing term of Eq. \eqref{2ndCHNSc}. On the one hand, this allows us to use a Picard iteration procedure on velocity to further  decouple the computation of the nonlinear Cahn-Hilliard equation from the linear Navier-Stokes equation, see Section 4 for details. On the other hand, owing to the special design, we are able to show the unconditionally unique solvability of the system \eqref{2ndCHNSa}-\eqref{2ndCHNSc} by a monotonicity argument (cf. Section 3). In fact, one can define a solution operator $\phi^{k+1}(\mu^{k+\12}): \mu^{k+\12} \rightarrow \phi^{k+1}$ from equation \eqref{2ndCHNSb}. Likewise, equation \eqref{2ndCHNSc} gives rise to a solution operator $\baru^{k+\12}(\mu^{k+\12}): \mu^{k+\12} \rightarrow \baru^{k+\12}$. As a result, the system \eqref{2ndCHNSa}-\eqref{2ndCHNSc} reduces to a scalar equation in terms of the unknown $\mu^{k+\12}$
\begin{align*}
\phi^{k+1}(\mu^{k+\12})-\phi^k + \delta t\nabla \cdot \big(\tphi^{k+\frac{1}{2}}\baru^{k+\frac{1}{2}}(\mu^{k+\12})\big) -\delta t\nabla \cdot \big(M(\tphi^{k+\frac{1}{2}})\nabla \mu^{k+\frac{1}{2}} \big)=0.
\end{align*}
The key here is to recognize that the left-hand side of the above equation defines a strictly monotone operator $T(\mu)$, in the sense that
\begin{align*}
\left\langle T(\mu)-T(\nu), \mu-\nu \right\rangle \geq 0,
\end{align*}
with equal sign if and only if $\mu=\nu$. Thus one can invoke the Browder-Minty Lemma \ref{BrowderMinty} (see Section 3) to prove the unique existence of such a solution $\mu^{k+\12}$. We remark that the variational approach \cite{Wise2010, CSW2013} is not directly applicable for the unique solvability of the Cahn-Hilliard-Navier-Stokes system \eqref{2ndCHNSa}-\eqref{2ndCHNSd}. In both cases (Cahn-Hilliard-Hele-Shaw, Cahn-Hilliard-Brinkman), the approach relies on the symmetry of the underlying systems which breaks down  in the Navier-Stokes equation due to the nonlinear advection.

\section{Properties of the scheme}
In this section, we summarize the properties of the discrete time, continuous space scheme \eqref{2ndCHNSa}-\eqref{2ndCHNSe}, namely mass-conservation, unconditional stability and unconditionally unique solvability. It will be clear from the proof, that these properties will be preserved when the scheme is combined with any consistent Galerkin type spatial discretization schemes.

First of all, one can readily obtain that the scheme is mass-conservative.
\begin{proposition}
The scheme \eqref{2ndCHNSa}-\eqref{2ndCHNSd} equipped with the boundary condition \eqref{2ndCHNSe} satisfies the mass-conservation, i.e.,
\begin{align*}
\int_{\Omega}\phi^{k+1}dx =\int_{\Omega}\phi^k dx, \quad k=0, 1, \cdots K-1.
\end{align*}
\end{proposition}

Next, we show that our numerical scheme \eqref{2ndCHNSa}-\eqref{2ndCHNSe} is unconditionally stable, thus allowing for large time stepping. Recall the definition of the total energy functional $E_{tot}(\mathbf{u}, \phi)$ in Eq. \eqref{Etot}.
\begin{proposition}
The scheme \eqref{2ndCHNSa}-\eqref{2ndCHNSd} with the boundary condition \eqref{2ndCHNSe} satisfies the modified energy law
\begin{align}\label{ModEnergyLaw}
&\Big\{E_{tot}(\mathbf{u}^{k+1}, \phi^{k+1})+\frac{\epsilon^{-1}}{4\sw}||\phi^{k+1}-\phi^k||\L^2+\frac{\delta t^2}{8}||\nabla p^{k+1}||\L^2\Big\} \nonumber\\
& -\Big\{E_{tot}(\mathbf{u}^{k}, \phi^{k})+\frac{\epsilon^{-1}}{4\sw}||\phi^{k}-\phi^{k-1}||\L^2+\frac{\delta t^2}{8}||\nabla p^{k}||\L^2\Big\}\nonumber \\
&=-\delta t\frac{\epsilon^{-1}}{\sw}||\sqrt{M}\nabla \mu^{k+\frac{1}{2}}||\L^2-\delta t \frac{1}{Re}||\nabla \baru^{k+\frac{1}{2}}||\L^2-\frac{\epsilon^{-1}}{4\sw}||\phi^{k+1}-2\phi^k+\phi^{k-1}||\L^2.
\end{align}
Thus it is \textbf{unconditionally stable}.
\end{proposition}

\begin{proof}
One first takes the $L^2$ inner product of Eq. \eqref{2ndCHNSa} with $\delta t \mu^{k+\frac{1}{2}}$ to obtain
\begin{align}\label{2ndCH1}
&\big(\phi^{k+1}-\phi^k, \mu^{k+\frac{1}{2}} \big)=-\delta t||\sqrt{M}\nabla \mu^{k+\frac{1}{2}}||\L^2+\delta t\big(\tphi^{k+\frac{1}{2}}\baru^{k+\frac{1}{2}},\nabla \mu^{k+\frac{1}{2}}\big).
\end{align}
Next, multiplying Eq. \eqref{2ndCHNSb} by $(\phi^{k+1}-\phi^k)$, performing integration by parts and using the the following identity
\begin{eqnarray}
& &\big(\tphi^{k+\frac{1}{2}},\phi^{k+1}-\phi^k\big)\nonumber \\
&=& \frac{1}{2}\big(3\phi^k-\phi^{k-1},\phi^{k+1}-\phi^k\big) \nonumber \\
&=&\frac{1}{2}\big(\phi^{k+1}+\phi^k,\phi^{k+1}-\phi^k\big) -\frac{1}{2}\big(\phi^{k+1}-2\phi^k+\phi^{k-1},\phi^{k+1}-\phi^k\big)\nonumber \\
&=&\frac{1}{2}\big(||\phi^{k+1}||\L^2-||\phi^k||\L^2\big)-\frac{1}{4}\{||\phi^{k+1}-\phi^k||\L^2-||\phi^k-\phi^{k-1}||\L^2\nonumber \\
& & +||\phi^{k+1}-2\phi^k+\phi^{k-1}||\L^2\}, \nonumber 
\end{eqnarray}
one deduces
\begin{align}\label{2ndCH2}
&-\big(\phi^{k+1}-\phi^k, \mu^{k+\frac{1}{2}}\big)+\big(f_0(\phi^{k+1})-f_0(\phi^k),1\big)+\frac{\epsilon^2}{2}(||\nabla \phi^{k+1}||\L^2-||\nabla \phi^k||\L^2)\nonumber \\
&+\frac{1}{4}\{||\phi^{k+1}-\phi^k||\L^2-||\phi^{k}-\phi^{k-1}||\L^2+||\phi^{k+1}-2\phi^k+\phi^{k-1}||\L^2\}=0,
\end{align}
where one has utilized the definition of $f_0(\phi)=\frac{1}{4}(\phi^2-1)^2$.
Summing up Eqs. \eqref{2ndCH1} and \eqref{2ndCH2} gives
\begin{align}\label{2ndCHf}
&\big(f_0(\phi^{k+1})-f_0(\phi^k),1\big)+\frac{\epsilon^2}{2}(||\nabla \phi^{k+1}||\L^2-||\nabla \phi^k||\L^2)+\frac{1}{4}(||\phi^{k+1}-\phi^k||\L^2-||\phi^{k}-\phi^{k-1}||\L^2)\nonumber \\
&=-\frac{1}{4}||\phi^{k+1}-2\phi^k+\phi^{k-1}||\L^2-\delta t||\sqrt{M}\nabla \mu^{k+\frac{1}{2}}||\L^2+\delta t\big(\tphi^{k+\frac{1}{2}}\baru^{k+\frac{1}{2}},\nabla \mu^{k+\frac{1}{2}}\big),
\end{align}

Now we turn to the Navier-Stokes part. Taking the $L^2$ inner product of Eq. \eqref{2ndCHNSc} with $\baru^{k+\frac{1}{2}}\delta t$ and using the skew-symmetry of the trilinear form $b$ in \eqref{trilinear}, one gets
\begin{align}\label{2ndNS1}
&\frac{1}{2}(||\baru^{k+1}||\L^2-||\bfu^k||\L^2)+\delta t \frac{1}{Re}||\nabla \baru^{k+\frac{1}{2}}||\L^2
=-\delta t \big(\nabla p^k, \baru^{k+\frac{1}{2}}\big)-\delta t \frac{\epsilon^{-1}}{\sw}\big(\tphi^{k+\frac{1}{2}} \nabla \mu^{k+\frac{1}{2}}, \baru^{k+\frac{1}{2}}\big).
\end{align}
Testing the first equation in \eqref{2ndCHNSd} by $\mathbf{u}^{k+1} \delta t$ and performing integration by parts yield
\begin{align}\label{2ndNS2}
&\frac{1}{2}(||\mathbf{u}^{k+1}||\L^2-||\baru^{k+1}||\L^2+||\mathbf{u}^{k+1}-\baru^{k+1}||\L^2)=0,
\end{align}
where one has utilized explicitly the divergence-free condition $\nabla \cdot \bfu^{k+1}=0$.
Next, we rewrite the projection step Eq. \eqref{2ndCHNSd} as
\begin{align*}
\frac{\mathbf{u}^{k+1}+\mathbf{u}^k-2\baru^{k+\frac{1}{2}}}{\delta t} + \frac{1}{2}\nabla(p^{k+1}-p^k)=0.
\end{align*}
Testing the above equation with $\frac{\delta t^2}{2}\nabla p^k$, one arrives at
\begin{align}\label{2ndNS3}
\frac{\delta t^2}{8}\big\{||\nabla p^{k+1}||\L^2-||\nabla p^{k}||\L^2-||\nabla(p^{k+1}-p^k)||\L^2\big\}=\delta t\big(\nabla p^k, \baru^{k+\frac{1}{2}}\big).
\end{align}
On the other hand, it follows directly from Eq. \eqref{2ndCHNSd} that 
\begin{align}\label{2ndNS4}
\frac{{\delta t}^2}{8}||\nabla(p^{k+1}-p^k)||\L^2=\frac{1}{2}||\mathbf{u}^{k+1}-\baru^{k+1}||\L^2.
\end{align}
Now summing up Eqs. \eqref{2ndNS1}-\eqref{2ndNS3} and in view of Eq. \eqref{2ndNS4}, one obtains
\begin{align}\label{2ndNSf}
&\frac{1}{2}(||\mathbf{u}^{k+1}||\L^2-||\mathbf{u}^k||\L^2)+\frac{\delta t^2}{8}\big\{||\nabla p^{k+1}||\L^2-||\nabla p^{k}||\L^2\big\}
\nonumber \\
&=-\delta t \frac{1}{Re}||\nabla \baru^{k+\frac{1}{2}}||\L^2-\delta t \frac{\epsilon^{-1}}{\sw}\big(\tphi^{k+\frac{1}{2}} \nabla \mu^{k+\frac{1}{2}}, \baru^{k+\frac{1}{2}}\big).
\end{align}

The energy law \eqref{ModEnergyLaw} then follows from summing up the multiple of Eq. \eqref{2ndCHf} by $\frac{\epsilon^{-1}}{\sw}$ and Eq. \eqref{2ndNSf}.
\end{proof}

\begin{remark}
Heuristically, $E_{tot}(\mathbf{u}^{k+1}, \phi^{k+1})+\frac{\tw}{4\epsilon}||\phi^{k+1}-\phi^k||\L^2+\frac{\delta t^2}{8}||\nabla p^{k+1}||\L^2$ is a second order approximation of $E_{tot}(\mathbf{u}^{k+1}, \phi^{k+1})$, as one can write
\begin{align*}
||\phi^{k+1}-\phi^k||\L^2=\delta t^2||(\phi^{k+1}-\phi^k)/\delta t||\L^2,
\end{align*}
and $(\phi^{k+1}-\phi^k)/\delta t$ is an approximation of $\phi_t$ at $t^{k+1}$.
\end{remark}

To prove the unconditionally unique solvability of Eqs. \eqref{2ndCHNSa}-\eqref{2ndCHNSe}, we write them in a weak form. Note that the pressure equation \eqref{2ndCHNSd} is completely decoupled from the rest of the equations. Thus one only needs to establish the unique solvability of Eqs. \eqref{2ndCHNSa}-\eqref{2ndCHNSc}. Once $\baru^{k+1}$ or equivalently $\baru^{k+\12}$ is known, one can find $\bfu^{k+1}$ and $p^{k+1}$ by either solving a Darcy problem as Eq. \eqref{2ndCHNSd} or solving a pressure Poisson equation and an update of the velocity as described in Eqs. \eqref{PrePoi}-\eqref{AupVel}. Hereafter, we denote by $L^2_0(\Omega)$ an $L^2$ subspace with mean zero, i.e., $L^2_0(\Omega):= \{f\in L^2(\Omega); \int_\Omega f dx=0\}$. 
\begin{definition}\label{def1}
Given that $\phi^k, \phi^{k-1} \in H^1(\Omega)$, $\mathbf{u}^k, \mathbf{u}^{k-1} \in \mathbf{H^1}(\Omega)$, and $p^k\in H^1(\Omega)\cap L_0^2(\Omega)$ for $k=1,2, \cdots K=[T/\delta t]$, the triple $\{\phi^{k+1}, \mu^{k+\frac{1}{2}}, \baru^{k+\frac{1}{2}}\}$ is said to be a weak solution of Eqs. \eqref{2ndCHNSa}-\eqref{2ndCHNSc} if they satisfy
\begin{align*}
&\phi^{k+1} \in H^1(\Omega), \quad \mu^{k+\frac{1}{2}} \in H^1(\Omega), 
\quad \baru^{k+\frac{1}{2}} \in \mathbf{H^1_0}(\Omega),
\end{align*}
and there hold, $\forall v\in H^1(\Omega), \varphi \in H^1(\Omega), \mathbf{v} \in \mathbf{H^1_0}(\Omega)$,
\begin{align}
&\big(\phi^{k+1}-\phi^k,v\big)+\delta t \big(M(\tphi^{k+\frac{1}{2}})\nabla \mu^{k+\frac{1}{2}}, \nabla v\big)-\delta t \big(\tphi^{k+\frac{1}{2}}{\baru}^{k+\frac{1}{2}}, \nabla v\big)=0, \label{2ndCHNSwa}\\
&\big(\mu^{k+\frac{1}{2}}, \varphi \big)=\frac{1}{4}\bigg([(\phi^{k+1})^2+(\phi^k)^2](\phi^{k+1}+\phi^k), \varphi\bigg)-\big(\tphi^{k+\frac{1}{2}}, \varphi \big) \nonumber \\
&+\frac{\epsilon^2}{2} \big(\nabla (\phi^{k+1}+\phi^k), \nabla \varphi \big),  \label{2ndCHNSwb}\\
&2\big({\baru}^{k+\frac{1}{2}}-\mathbf{u}^k, \mathbf{v} \big)+\delta t \frac{1}{Re}\big( \nabla{\baru}^{k+\frac{1}{2}}, \nabla\mathbf{v}\big)+\delta t b(\tu^{k+\frac{1}{2}}, {\baru}^{k+\frac{1}{2}}, \mathbf{v}) \nonumber \\
& =-\delta t\Big(\nabla p^k, \mathbf{v}\Big)-\delta t \frac{\epsilon^{-1}}{\sw} \big(\tphi^{k+\frac{1}{2}} \nabla \mu^{k+\frac{1}{2}}, \mathbf{v}\big),  \label{2ndCHNSwc}
\end{align}
\end{definition}
where the trilinear form $b$ is defined in \eqref{trilinear}. 


We will mainly use the well-known Browder-Minty lemma in establishing the unconditionally unique solvability of Eqs. \eqref{2ndCHNSwa}-\eqref{2ndCHNSwc}, cf. \cite{ReRo2004}, p.364, Theorem 10.49.
\begin{lemma}[Browder-Minty]\label{BrowderMinty}
Let $X$ be a real, reflexive Banach
space and let $T : X \rightarrow X^\prime$ (the dual space of $X$) be bounded, continuous, coercive and monotone.
Then for any $g \in X^\prime$ there exists a solution $u \in X$ of the equation
\begin{align*}
T(u)=g.
\end{align*}
If further, the operator $T$ is strictly monotone, then the solution $u$ is unique.
\end{lemma}

We observe that Eqs. \eqref{2ndCHNSwa}-\eqref{2ndCHNSwc} are coupled together through $\mu^{k+\12}$. It is possible to rewrite the system equivalently as a scalar equation in terms of unknown $\mu^{k+\12}$. To do so, we introduce two solution operators $\phi^{k+1}(\mu^{k+\12}): \mu^{k+\12} \rightarrow \phi^{k+1}$ and $\baru^{k+\12}(\mu^{k+\12}): \mu^{k+\12} \rightarrow \baru^{k+\12}$ by solving equations \eqref{2ndCHNSwb} and \eqref{2ndCHNSwc}, respectively, for a given source function $\mu^{k+\frac{1}{2}} \in H^1(\Omega)$. Specifically, one can establish the following lemmas.

\begin{lemma}[solvability of Eq. \eqref{2ndCHNSwb}]\label{wellposed2nd}
Given a source function $\mu^{k+\frac{1}{2}} \in H^1(\Omega)$ and known functions $\phi^k, \phi^{k-1} \in H^1(\Omega)$, there exists a unique solution $\phi^{k+1} \in H^1(\Omega)$ to Eq. \eqref{2ndCHNSwb}. Moreover, the solution is bounded and depends continuously on $\mu^{k+\frac{1}{2}}$ in the weak topology.
\end{lemma}

\begin{lemma}[solvability of Eq. \eqref{2ndCHNSwc}]\label{wellposed3rd}
Given a source function $\mu^{k+\frac{1}{2}} \in H^1(\Omega)$, known functions $\phi^k, \phi^{k-1} \in H^1(\Omega)$ and $\mathbf{u}^k, \mathbf{u}^{k-1} \in \bfH^1(\Omega)$, there exists a unique solution $\baru^{k+\frac{1}{2}} \in \mathbf{H}^1_0(\Omega)$ to Eq. \eqref{2ndCHNSwc}. In addition, the solution is bounded and depends continuously on $\mu^{k+\frac{1}{2}}$ in the strong topology.
\end{lemma}

It will be clear from the proof of Proposition \ref{Solva} below that the unique solvability of Eq. \eqref{2ndCHNSwb} can be proved by using Browder-Minty Lemma \ref{BrowderMinty} as well.   The boundedness and continuity of the solution readily follow from the fact that Eq. \eqref{2ndCHNSwb} is a semilinear elliptic equation for $\phi^{k+1}$ with cubic nonlinearity. Lemma \ref{wellposed3rd} can be proved by invoking the Lax-Milgram Theorem. We omit the details here for conciseness.

With the help of Lemma \ref{BrowderMinty}, Lemma \ref{wellposed2nd} and Lemma \ref{wellposed3rd}, one can prove the unique existence of a weak solution in the sense of Definition \ref{def1}.
\begin{proposition}\label{Solva}
Assume that $\phi^k, \phi^{k-1} \in H^1(\Omega)$, $\mathbf{u}^k, \mathbf{u}^{k-1} \in \bfH^1(\Omega)$, and $p^k \in H^1(\Omega)$  are known functions for $k=1,2, \cdots, K-1$. Then there exists a unique weak solution to Eqs. \eqref{2ndCHNSa}-\eqref{2ndCHNSc} in the sense of Def. \ref{def1}
\end{proposition}
\begin{proof}
Here for notational simplicity, we will temporarily omit the the superscripts on $\phi^{k+1}, \mu^{k+\12}, \baru^{k+\12}$.

For any $\mu \in H^1(\Omega)$, one defines an operator $T: H^1(\Omega) \rightarrow (H^1(\Omega))^\prime$ such that
\begin{align}\label{Operator}
\langle T(\mu), v \rangle := \big(\phi-\phi^k,v\big)+\delta t \big(M\nabla \mu, \nabla v\big)-\delta t \big(\tphi^{k+\frac{1}{2}}\baru, \nabla v\big), \quad \forall v \in H^1(\Omega),
\end{align}
where $\langle, \rangle$ is the duality pairing between $(H^1(\Omega))^\prime$ and $H^1(\Omega)$,  $\phi$ and $\baru$ are the unique solutions to Eqs. \eqref{2ndCHNSwb} and \eqref{2ndCHNSwc}  that are defined in Lemma \ref{wellposed2nd} and Lemma \ref{wellposed3rd}, respectively.

It readily follows that
\begin{align*}
|\langle T(\mu), v \rangle|  \leq C(\delta t)\big( ||\phi||_{L^2} +||\phi^k||_{L^2}  +||\nabla \mu||\L +||\tphi^{k+\frac{1}{2}}||\H||\baru||\H\big)||\mathbf{v}||\H,
\end{align*}
where we have used the boundedness of the mobility function $m_1 \leq M \leq m_2$ for constants $0<m_1 \leq m_2$.
Thus the boundedness of the operator $T$ follows from the  boundedness of $\phi$ and $\baru$ as functions of $\mu$ in Lemmas \ref{wellposed2nd} and \ref{wellposed3rd}. Similarly, one can verify that the operator $T: H^1(\Omega)\rightarrow (H^1(\Omega))^\prime$ is continuous as a consequence of the continuity of $\phi$ and $\baru$ on $\mu$.

For the monotonicity, one obtains from the definition of $T$ in \eqref{Operator}
\begin{align}\label{Mono}
&\langle T(\mu)-T(\nu), \mu-\nu \rangle = \big(\phi(\mu)-\phi(\nu),\mu-\nu\big)+\delta t ||\sqrt{M}\nabla(\mu-\nu)||\L^2 \nonumber\\
&-\delta t \big(\tphi^{k+\frac{1}{2}}[\baru(\mu)-\baru(\nu)], \nabla (\mu-\nu)\big),   \quad \forall \mu, \nu \in H^1(\Omega),
\end{align}
where $\phi(\nu)$ and $\baru(\nu)$ are solutions to Eqs. \eqref{2ndCHNSwb} and \eqref{2ndCHNSwc}, respectively, with a given source function $\nu$.
For the first term on the right hand side of \eqref{Mono}, one subtracts Eq. \eqref{2ndCHNSwb} with source functions $\mu$ and $\nu$ respectively to get
\begin{align*}
\big(\mu-\nu, \varphi\big)&=\frac{1}{4}\int_\Omega(\phi(\mu)-\phi(\nu))[(\phi(\mu)+\phi(\nu))^2+(\phi(\mu)+\phi^k)^2+(\phi^k+\phi(\nu))^2]\varphi\, dx \\
&+\frac{\epsilon^2}{2}\big(\nabla(\phi(\mu)-\phi(\nu)), \nabla \varphi\big), \quad \forall \varphi \in H^1(\Omega).
\end{align*}
By taking $\varphi=\phi(\mu)-\phi(\nu)$ in the above equation, one concludes that
\begin{align}\label{Mono1st}
\big(\mu-\nu, \phi(\mu)-\phi(\nu)\big) \geq 0,
\end{align}
 and that the equality holds if only if $\mu=\nu$ thanks to the uniqueness of solutions to Eq. \eqref{2ndCHNSwb} in Lemma \ref{wellposed2nd}.
 By the linearity of Eq. \eqref{2ndCHNSwc}, the third term on the right hand side of  Eq. \eqref{Mono} can be written as
 \begin{align}\label{Mono2nd}
 -\delta t \big(\tphi^{k+\frac{1}{2}}[\baru(\mu)-\baru(\nu)], \nabla (\mu-\nu)\big)=\epsilon\sw\{2||\baru(\mu)-\baru(\nu)||\L^2+\frac{\delta t}{Re}||\nabla(\baru(\mu)-\baru(\nu))||\L^2\},
 \end{align}
 where the convective term vanishes thanks to the skew-symmetry of the form $b$. In view of \eqref{Mono1st} and \eqref{Mono2nd}, one sees 
 \begin{align}
 \langle T(\mu)-T(\nu), \mu-\nu \rangle \geq 0,
 \end{align}
 with equality if only if $\mu=\nu$. This establishes the strict monotonicity of the operator $T$.
 
We next turn to the coercivity of the operator $T$. One has
\begin{align}\label{Coer}
\langle T(\mu), \mu \rangle = \big(\phi-\phi^k,\mu\big)+\delta t \big(M\nabla \mu, \nabla \mu\big)-\delta t \big(\tphi^{k+\frac{1}{2}}\baru, \nabla \mu\big), \quad \forall \mu \in H^1(\Omega).
\end{align} 
Taking the test function $\varphi=\phi-\phi^k$ in Eq. \eqref{2ndCHNSwb}, one obtains
\begin{align}\label{Coer1st}
&\big(\phi-\phi^k, \mu\big) \nonumber \\
&=\frac{1}{4} \int_\Omega \phi^4-(\phi^k)^4\, dx -\int_\Omega \tphi^{k+\frac{1}{2}}(\phi-\phi^k)\, dx+\frac{\epsilon^2}{2}\int_\Omega |\nabla \phi|^2-|\nabla \phi^k|^2\, dx. \nonumber \\
&\geq \frac{1}{8} \int_\Omega \phi^4\, dx +\frac{\epsilon^2}{2}\int_\Omega |\nabla \phi|^2\,dx-C(\epsilon, \Omega)\big(||\tphi^{k+\frac{1}{2}}||\L^2+||\phi^k||\H+1\big)
\end{align}
Similarly, one can take the test function  $\mathbf{v}=\baru$ in Eq. \eqref{2ndCHNSwc} to get
\begin{align}\label{Coer2nd}
-\delta t \big(\tphi^{k+\frac{1}{2}}\baru, \nabla \mu\big)&=\epsilon \sw \{2||\baru||\L^2+\delta t ||\nabla \baru||\L^2- \delta t \big(2\mathbf{u}^k-\nabla p^k, \baru\big)\} \nonumber \\
&\geq C(\epsilon, \sw, \delta t)\{||\baru||\L^2+||\nabla \baru||\L^2-(||\mathbf{u}^k||\L^2+||\nabla p^k||\L^2)\}.
\end{align}
 Collecting inequalities \eqref{Coer1st} and \eqref{Coer2nd}, one finds that Eq. \eqref{Coer} becomes
\begin{align}\label{Coer2}
\langle T(\mu), \mu \rangle \geq C||\nabla \mu||\L^2+\frac{1}{8} ||\phi||_{L^4}^4 +\frac{\epsilon^2}{2} ||\nabla \phi||^2\L+C(||\baru||\L^2+||\nabla \baru||\L^2)- C,
\end{align}
where again the boundedness of the mobility function has been invoked.
To have coercivity in $H^1(\Omega)$, one needs to bound the average $m(\mu):=\frac{1}{|\Omega|}\int_\Omega \mu\, dx$ appropriately. For this, one takes the test function $\varphi=1$ in Eq. \eqref{2ndCHNSwb}.
\begin{align*}
\Big|\int_\Omega \mu\, dx\Big| &\leq \frac{1}{4}\int_\Omega |\phi|^3+|\phi|| \phi^k|^2+|\phi^k \phi^2|\, dx+ C(||\phi^k||_{L^3}^3+||\tphi^{k+\frac{1}{2}}||\H) \\
& \leq C(||\phi||_{L^3}^3+||\phi||\L^2+||\phi^k||\L||\phi||_{L^4}^2)+C \nonumber \\
& \leq C(||\phi||_{L^3}^3+||\phi||\L^2+||\phi||_{L^4}^3)+C \nonumber \\
&\leq C(||\phi||_{L^4}^3+||\phi||_{L^4}^2)+ C,
\end{align*}
where one has applied Young's inequality.
It readily follows that 
\begin{align}\label{Coer21}
|m(\mu)|^\frac{4}{3} \leq C||\phi||_{L^4}^4+ C
\end{align}
Thus by using Poincar\'{e} inequality, one gets from \eqref{Coer2} and \eqref{Coer21} that
\begin{align}
\langle T(\mu), \mu \rangle \geq C||\mu||\H^{\frac{4}{3}}-C,
\end{align}
which implies the coercivity of $T$.

Now Browder-Minty Lemma \ref{BrowderMinty} yields that there exists a unique solution $\mu^{\star} \in H^1(\Omega)$ such that $\langle T(\mu^{\star}), v \rangle=0, \forall v \in H^1(\Omega)$. In view of the definition of $T$ in \eqref{Operator}, one sees $\mu^\star \in H^1(\Omega)$ and the corresponding $\phi^{\star} \in H^1(\Omega), \baru^{\star} \in \mathbf{H}_0^1$ uniquely solve the system \eqref{2ndCHNSwa}-\eqref{2ndCHNSwc}. 
\end{proof}

\section{Mixed Finite Element Formulation}
We now discretize the time-discrete scheme \eqref{2ndCHNSa}-\eqref{2ndCHNSe} in space by finite element method.
Let $\mathcal{T}_h$ be a quasi-uniform triangulation of the domain $\Omega$ of mesh size $h$.  We introduce $\mathbf{X}_h$ and $Y_h$ the finite element approximations of $\mathbf{H}_0^1(\Omega)$ and $H^1(\Omega)$ respectively based on the triangulation $\mathcal{T}_h$. In addition, we define $M_h=Y_h \cap L^2_0(\Omega):= \{q_h \in Y_h; \int_{\Omega}q_h dx=0\}$. We assume that $Y_h \times Y_h$ is a stable pair for the biharmonic operator in the sense that there holds the inf-sup condition
\begin{align*}
\sup_{\phi_h \in Y_h} \frac{(\nabla \phi_h, \nabla \varphi_h)}{||\phi_h||_{H^1}} \geq c||\varphi_h||_{H^1}, \quad \forall \varphi_h \in Y_h.
\end{align*}
We also assume that $\mathbf{X}_h$ and $Y_h$ are stable approximation spaces for velocity and pressure in the sense of
\begin{align*}
\sup_{\mathbf{v}_h \in \mathbf{X}_h} \frac{(\nabla \cdot \mathbf{v}_h,  q_h)}{||\mathbf{v}_h||_{H^1}} \geq c||q_h||_{L^2}, \quad \forall q_h \in Y_h.
\end{align*}
It is pointed out \cite{GuQu1998b} that the inf-sup condition is necessary for the stability of pressure even though one may solve the projection step as a pressure Poisson equation.

Then the fully discrete finite element formulation for scheme \eqref{2ndCHNSa}-\eqref{2ndCHNSe} reads: find \linebreak $(\phi_h^{k+1}, \mu^{k+\frac{1}{2}}_h, \baru_h^{k+\frac{1}{2}}, p_h^{k+1}, \mathbf{u}_h^{k+1}) \in Y_h \times Y_h \times \mathbf{X}_h \times M_h \times \mathbf{X}_h$ such that for all $(v_h, \varphi_h, \mathbf{v}_h, q_h) \in Y_h \times Y_h \times \mathbf{X}_h \times Y_h$ there hold
\begin{align}
&\big(\phi^{k+1}_h-\phi^k_h,v_h\big)+\delta t \big(M\nabla \mu^{k+\frac{1}{2}}_h, \nabla v_h\big)-\delta t \big(\tphi^{k+\frac{1}{2}}_h\baru^{k+\frac{1}{2}}_h, \nabla v_h\big)=0, \label{d2ndCHNSwa}\\
&\big(\mu^{k+\frac{1}{2}}_h, \varphi_h \big)=\frac{1}{4}\bigg([(\phi^{k+1}_h)^2+(\phi^k_h)^2](\phi^{k+1}_h+\phi^k_h), \varphi_h\bigg)-\big(\tphi^{k+\frac{1}{2}}_h, \varphi_h \big) \nonumber\\
&+\frac{\epsilon^2}{2} \big(\nabla (\phi^{k+1}_h+\phi^k_h), \nabla \varphi_h \big),   \label{d2ndCHNSwb}\\
&\big(2\baru^{k+\frac{1}{2}}_h, \mathbf{v}_h \big)+\delta t \frac{1}{Re}\big( \nabla\baru^{k+\frac{1}{2}}_h, \nabla \mathbf{v}_h\big)+\delta t b\big(\tu^{k+\frac{1}{2}}_h, \baru^{k+\frac{1}{2}}_h, \mathbf{v}_h\big) =-\delta t\big(\nabla p^k_h, \mathbf{v}_h\big) \nonumber \\
&+\big(2\mathbf{u}^k_h, \mathbf{v}_h \big)-\delta t \frac{\epsilon^{-1}}{\sw} \big(\tphi^{k+\frac{1}{2}}_h \nabla \mu^{k+\frac{1}{2}}_h, \mathbf{v}_h\big), \label{d2ndCHNSwc}  \\
&\big(\bfu^{k+1}_h-\tu^{k+1}_h, \mathbf{v}_h\big)+ \frac{\delta t}{2}\big(\nabla (p_h^{k+1}-p_h^k), \mathbf{v}_h\big)=0, \label{d2ndCHNSwd} \\
&\big(\nabla \cdot \bfu^{k+1}_h, q_h\big)=0.  \label{d2ndCHNSwe}
\end{align}
The notations used here are defined in \eqref{notas} and Eq. \eqref{trilinear}. 

The properties of the time-discrete scheme \eqref{2ndCHNSa}-\eqref{2ndCHNSe} (i.e., mass-conservation, unconditional stability and unconditionally unique solvability) are preserved by the fully discrete formulation \eqref{d2ndCHNSwa}-\eqref{d2ndCHNSwe}. Note that Eqs. \eqref{d2ndCHNSwd}-\eqref{d2ndCHNSwe} amount to solving the projection step \eqref{2ndCHNSd}-\eqref{2ndCHNSe} as a Darcy problem. This formulation is shown \cite{GuQu1998a} to yield an optimal condition number for the pressure operator associated with finite element spatial discretizations. An alternative way of solving Eqs. \eqref{d2ndCHNSwd}-\eqref{d2ndCHNSwe} is the so-called "approximate projection" (cf. \cite{Codina2001} and references therein)
\begin{align*}
&\big(\nabla(p^{k+1}_h-p^k_h), \nabla q_h\big)=\frac{2}{\delta t} \big(\tu^{k+1}_h, \nabla q_h \big),  \forall q_h \in Y_h\\
&\big(\bfu_h^{k+1}, \mathbf{v}_h\big)=\big(\tu_h^{k+1}-\frac{\delta t}{2}\nabla(p_h^{k+1}-p_h^k), \mathbf{v}_h\big), \forall \mathbf{v}_h \in \mathbf{X}_h. 
\end{align*} 
One can still prove the unconditional stability of the scheme with the approximate projection, see the reference above. In our numerical experiment, we observe the $L^2$ error of the pressure is indeed smaller in the former case, though at the expense of more memory consumed due to the coupling between the velocity and pressure.
  


Note that the only nonlinear term appears in the chemical potential equation \eqref{d2ndCHNSwb}. We thus adopt a Picard iteration procedure on velocity to decouple the computation of the nonlinear Cahn-Hilliard equation \eqref{d2ndCHNSwa} and \eqref{d2ndCHNSwb} from that of the linear Navier-Stokes equation \eqref{d2ndCHNSwc}. Denote by $i$ the Picard iteration index. Specifically, given the velocity $\baru^{k+\12, i}$, we solve for $\phi^{k+1, i+1}, \mu^{k+\12, i+1}$ from the Cahn-Hilliard equation \eqref{d2ndCHNSwa} -- \eqref{d2ndCHNSwb} by Newton's method. As $\mu^{k+\12, i+1}$ is available, we can then proceed to solve for $\baru^{k+\12, i+1}$ from the linear equations \eqref{d2ndCHNSwc}. We repeat this procedure until the relative difference between two iterations within a fixed tolerance. We summarize this procedure in four steps  as follows:

\noindent \textbf{Step $1$}: given $\baru^{k+\12, i}$, find $(\phi^{k+1, i+1}_h, \mu^{k+\12, i+1}_h) \in Y_h \times Y_h$ such that $\forall (v_h, \varphi_h) \in Y_h \times Y_h$
\begin{align*}
&  \big(\phi^{k+1, i+1}_h-\phi^k_h,v_h\big)+\delta t \big(M\nabla \mu^{k+\frac{1}{2}, i+1}_h, \nabla v_h\big)-\delta t \big(\tphi^{k+\frac{1}{2}}_h\baru^{k+\12, i}_h, \nabla v_h\big)=0, \\
& \big(\mu^{k+\frac{1}{2}, i+1}_h, \varphi_h \big)=\frac{1}{4}\bigg([(\phi^{k+1, i+1}_h)^2+(\phi^k_h)^2](\phi^{k+1, i+1}_h+\phi^k_h), \varphi_h\bigg)-\big(\tphi^{k+\frac{1}{2}}_h, \varphi_h \big) \\
&+\frac{\epsilon^2}{2} \big(\nabla (\phi^{k+1, i+1}_h+\phi^k_h), \nabla \varphi_h \big), 
\end{align*}

\noindent \textbf{Step $2$}: find $\baru^{k+\12, i+1} \in \mathbf{X}_h$ such that $\forall \mathbf{v}_h \in \mathbf{X}_h$
\begin{align*}
&\big(2\baru^{k+\frac{1}{2}, i+1}_h, \mathbf{v}_h \big)+\delta t \frac{1}{Re}\big( \nabla\baru^{k+\frac{1}{2}, i+1}_h, \nabla \mathbf{v}_h\big)+\delta t b\big(\tu^{k+\frac{1}{2}}_h, \baru^{k+\frac{1}{2}, i+1}_h, \mathbf{v}_h\big) =-\delta t\big(\nabla p^k_h, \mathbf{v}_h\big) \\
&+\big(2\mathbf{u}^k_h, \mathbf{v}_h \big)-\delta t \frac{\epsilon^{-1}}{\sw}\big(\tphi^{k+\frac{1}{2}}_h \nabla \mu^{k+\frac{1}{2}, i+1}_h, \mathbf{v}_h\big), 
\end{align*}

\noindent \textbf{Step $3$}: find $\baru^{k+\12}_h \in \mathbf{X}_h$ by repeating Step 1 and Step 2 until the relative error in $L^2$ of $(\baru^{k+\12, i+1}_h-\baru^{k+\12, i}_h)$ is within a fixed tolerance.

\noindent \textbf{Step $4$}: find $\bfu_h^{k+1} \in \mathbf{X}_h, p_h^{k+1} \in Y_h$ (equivalently, $p_h^{k+1}-p_h^{k}$) such that $\forall \mathbf{v}_h \in \mathbf{X}_h, q_h \in Y_h$
\begin{align*}
&\big(\bfu^{k+1}_h-\tu^{k+1}_h, \mathbf{v}_h\big)+ \frac{\delta t}{2}\big(\nabla (p_h^{k+1}-p_h^k), \mathbf{v}_h\big)+\big(\nabla \cdot \bfu^{k+1}_h, q_h\big)=0.
\end{align*}
We remark that our scheme is a two step method.
One can solve for $\phi^1_h, \mu^1_h, \bfu^1_h, p^1_h$ through a coupled first order scheme (see, for example, \cite{Feng2006, KSW2008}) to initialize the second order scheme. Numerical simulations in \cite{KKL2004} suggest that at least 4 grid elements across the interfacial region of thickness $\sqrt{2}\epsilon$ are needed for accuracy. To improve the efficiency of the algorithm, we explore the capability of adaptive mesh refinement of FreeFem++ (cf. \cite{Hecht2012}) in which  a variable metric/Delaunay automatic meshing algorithm is implemented.

\section{Numerical Experiments}
In this section, we perform some standard tests to gauge our numerical algorithm. For simplicity, we will use P1--P1 function spaces for $Y_h \times Y_h$ , and P1b--P1 mixed finite element spaces for $\mathbf{X}_h \times Y_h$ . It is well-known \cite{Ciarlet2002, ABF1984} that these approximation spaces satisfy the inf-sup conditions for the biharmonic operator and Stokes operator, respectively. In principle, any inf-sup compatible approximation spaces for biharmonic operator and Stokes operator can be used, for example,  P2--P2 for $Y_h \times Y_h$, and Taylor-Hood P2--P1 for $\mathbf{X}_h \times Y_h$.

\subsection{Convergence, energy dissipation, mass conservation}
Here we provide some numerical evidence to show that our scheme is second order accurate, energy-dissipative and mass-conservative.

As the Cahn-Hilliard equation does not have a natural forcing term, we verify the second order convergence of the scheme by a Cauchy convergence test. We consider the problem in a unit square domain $\Omega=[0,1]\times[0,1]$. The initial conditions are taken to be
\begin{align*}
& \phi_0=0.24\cos(2\pi x)\cos(2\pi y)+0.4\cos(\pi x)\cos(3\pi y), \\
& \bfu_0=(-\sin(\pi x)^2\sin(2\pi y), \sin(\pi y)^2\sin(2\pi x)).
\end{align*}
We impose no-slip no penetration boundary conditions for velocity, and homogeneous Neumann boundary condition for $\phi$ and $\mu$ .

The final time is $T=0.1$, the grid in space is uniform $h=\frac{\sqrt{2}}{2^n}$ ($2^n$ grid points in each direction),  for $n$ from $5$ to $9$, and the refinement path is taken to be $\delta t= \frac{0.2}{\sqrt{2}}h$. The other parameters are $\epsilon=0.04$, $M=0.1$, $\sw=25$ , $Re=100$. We calculate the the rate at witch the Cauchy difference converges to zero in the $L^2$ norm. The errors and convergence rates are given in Table \ref{CauchyP1} . The results show that the scheme is of second order accuracy for $\phi$ and $\bfu$ in $L^2$ norm, and the rate of convergence for pressure $p$ appear to be only first order.

\begin{table}[h!]
 \begin{center}
\begin{tabular}{c c c c c c c c}
 \hline 
  & $32-64$ & rate & $64-128$ & rate • &$128-256$  & rate  & $256-512$ \\ 
 \hline 
 $\phi$• & $4.14e-3$ & $1.90$ & $1.11-3$ & $1.97$ & $2.83e-4$ & $1.99$& $7.12e-5$ \\ 

 $u$ & $7.21e-4$• &$2.08$  &$1.70e-4$  &$  2.04$  & $4.16e-5$ &$ 2.02$ & $1.03e-5$ \\ 

 $v$ & $6.99e-4$ & $ 2.11 $ & $1.62e-4$ & $ 2.05$ &$3.93e-5$  & $2.02$ & $9.71e-6$ \\ 
 
 $p$ & $2.05e-3$ & $1.75$ & $6.10e-4$& $ 1.62$ & $1.98e-4$ & $ 1.44$ &$7.27e-5$ \\ 
 \hline 
 \end{tabular}  
 \caption{Cauchy convergence test; errors are measured in $L^2$ norm; $2^n$ grid points in each direction for $n$ from $5$ to $9$,  $\delta t= \frac{0.2}{2}h$, $\epsilon=0.04$, $M=0.1$, $\sw=25$ , $Re=100$.}
 \label{CauchyP1}
  \end{center}
  \end{table}

Next, we verify numerically that the total energy of the system is non-increasing at each time step. We define two discrete energy functional at discrete time $t=k \delta t$ according to Proposition \ref{ModEnergyLaw}
\begin{align*}
&E^{h,t}=\int_{\Omega}\frac{1}{2}|\mathbf{u}_h^k|^2\, dx+ \frac{1}{\sw}\int_{\Omega} \big( \frac{1}{\epsilon} f_0(\phi_h^k)+\frac{\epsilon}{2}|\nabla \phi_h^k|^2\big)\, dx, \\
&E^{h,t}_{app}=E^{h,t}+\frac{\epsilon^{-1}}{4\sw}\int_{\Omega}|\phi^{k}_h-\phi^{k-1}_h|^2 dx+\frac{\delta t^2}{8}\int_{\Omega}|\nabla p^{k}_h|^2 dx.
\end{align*}
 In the calculation, we take $\delta t= 0.005$, $h=\frac{\sqrt{2}}{128}$ and a constant mobility $M=1.0$. The other parameters are the same as ones in the Cauchy convergence test. Fig. \ref{DEn} shows that both of the discrete energy functional $E^{h,t}$ and $E^{h,t}_{app}$ are indeed non-increasing in time. Moreover, since $E^{h,t}_{app}$ is a second order approximation of $E^{h,t}$ in terms of $\delta t$, the qualitative evolution behaviour of $E^{h,t}$ and $E^{h,t}_{app}$ is virtually the same.

\begin{figure}
        \centering
        \begin{subfigure}[b]{0.5\textwidth}
                \includegraphics[width=\textwidth]{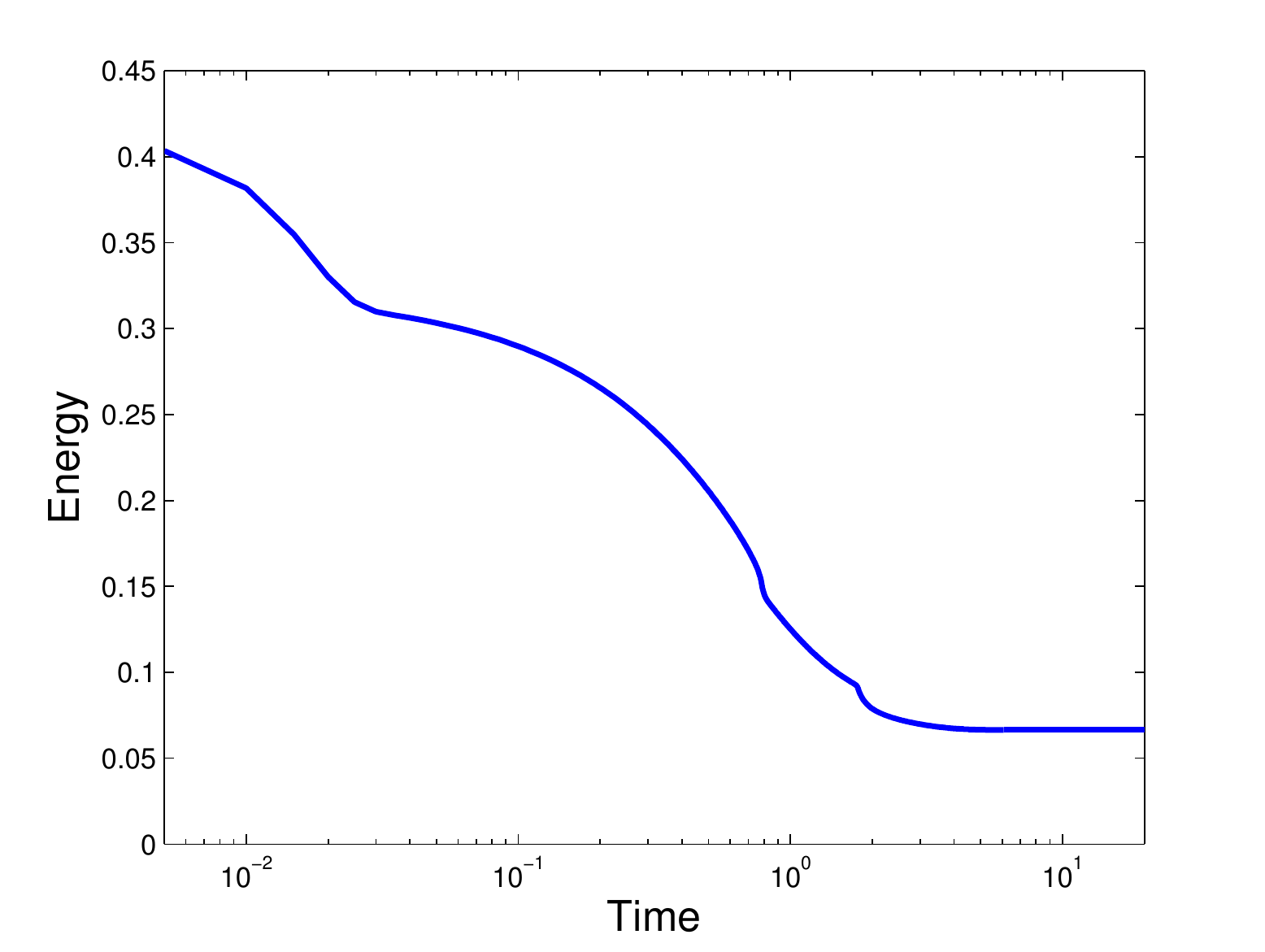}
                \caption{Evolution of $E^{h,t}$}
                \label{DEnTrue}
        \end{subfigure}%
        ~ 
        \begin{subfigure}[b]{0.5\textwidth}
                \includegraphics[width=\textwidth]{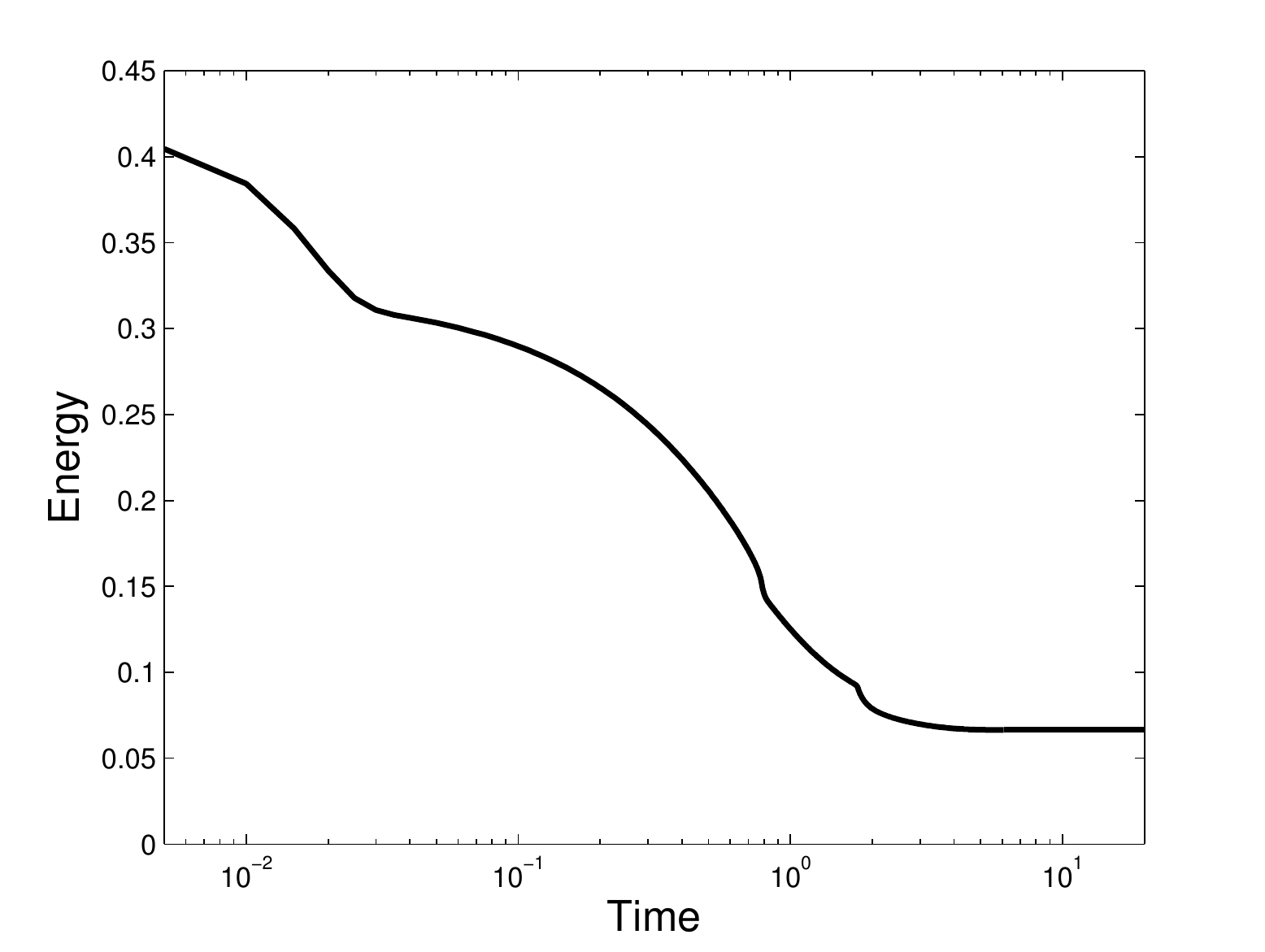}
                \caption{Evolution of $E^{h,t}_{app}$}
                \label{DEnNum}
        \end{subfigure}
        \caption{Time evolution of the discrete energy; $\delta t= 0.005$, $h=\frac{\sqrt{2}}{128}$, $M=1.0$, $\epsilon=0.04$,  $\sw=25$ , $Re=100$.}
        \label{DEn}
\end{figure}

In Fig. \ref{Mass},  we show the time evolution of the discrete mass $\int_{\Omega}\phi_h^k dx$ associated with the energy dissipation test (Fig. \ref{DEn}).   Note that $\int_{\Omega}\phi_0 dx=0$. After projection into the finite element space P1, we have $\int_{\Omega} \phi_h^0dx =8.14e-06$.  Fig. \ref{Mass} shows that the exact value is preserved during the evolution, which verifies that our scheme is conservative.

\begin{figure}[h!]
\centering
 \includegraphics[width=0.5\linewidth]{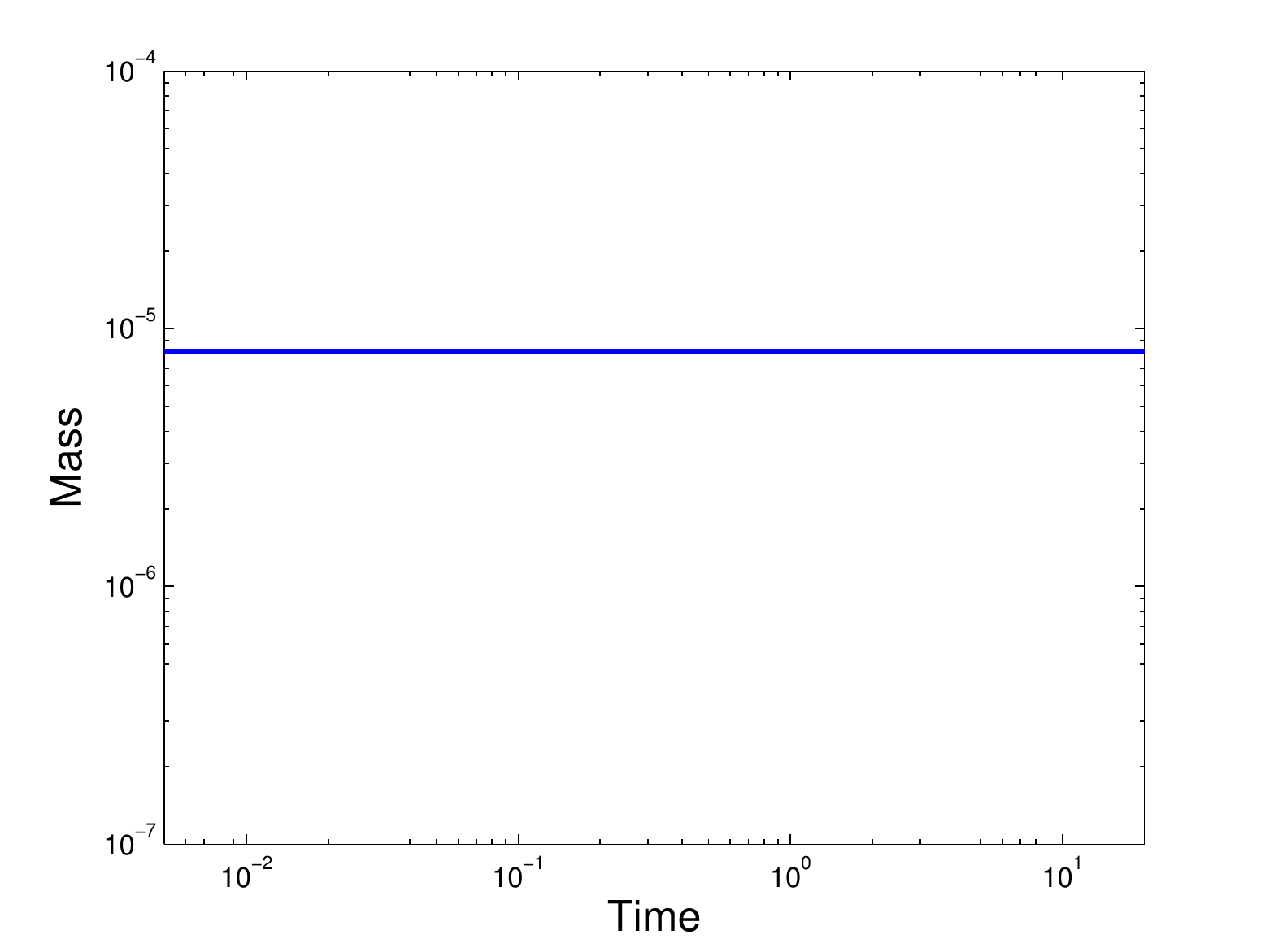}
 \caption{Time evolution of the discrete mass $\int_\Omega \phi_h^k dx$; the parameters are given in Fig. \ref{DEn}.}
  \label{Mass}
\end{figure}

\subsection{Shape relaxation}
Here we use the CHNS system \eqref{CH}-\eqref{div} to simulate the relaxation of an isolated shape in a two-phase flow system. The initial shape is a small square located in the middle of the domain (cf. Fig. \ref{IniSh}). For velocity, we set both the initial condition and boundary condition to be zero. We impose homogeneous Neumann boundary conditions fro $\phi$ and $\mu$. The parameters are $\epsilon=0.005$, $\sw=200$, $M(\phi)=0.1\sqrt{(1-\phi^2)^2+\epsilon^2}$, $Re=10$, $\delta t=0.005$. In space, we explore the adaptive mesh refinement of FreeFem++ (cf. \cite{Hecht2012}) which uses a variable metric/Delaunay automatic meshing algorithm. Specifically, we adapt the mesh according to the Hessian of the order parameter such that at least four grid cells are located across the diffuse interface.

\begin{figure}[h!]
\centering
 \includegraphics[width=0.3\linewidth]{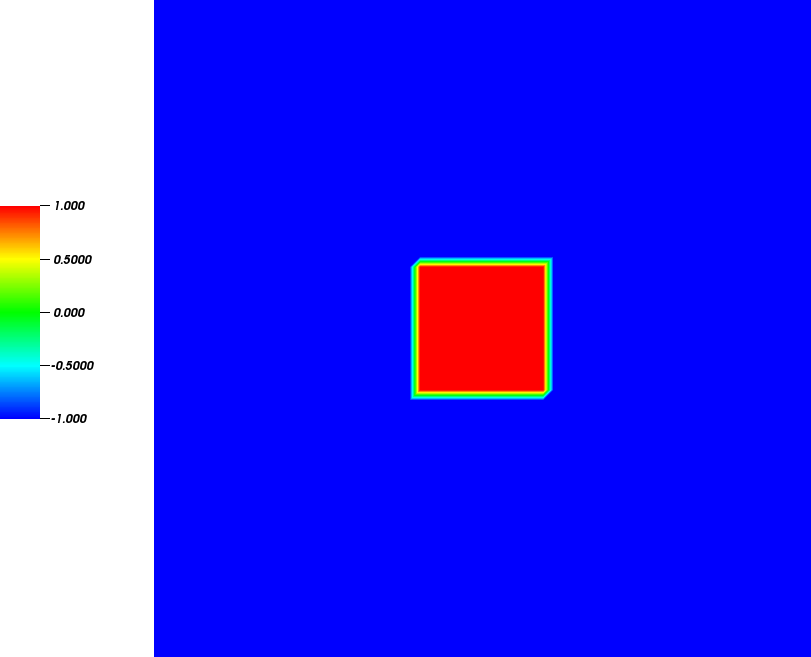}
 \caption{The initial shape of the order parameter for simulations of shape relaxation.}
  \label{IniSh}
\end{figure}

Since the initial velocity is zero, the initial total energy of the system is the surface energy.   Due to the effect of surface tension and the isotropy of the mobility, isolated irregular shape will relax to a circular shape. This relaxation is observed in Fig. \ref{ShaRE}. We also show the effectiveness of the adaptive mesh refinement at $t=0.02$ and $t=0.4$ in Fig. \ref{AdpM}.

\begin{figure}[h!]
\centering
\begin{tabular}{cc}
   \includegraphics[width=0.3\linewidth]{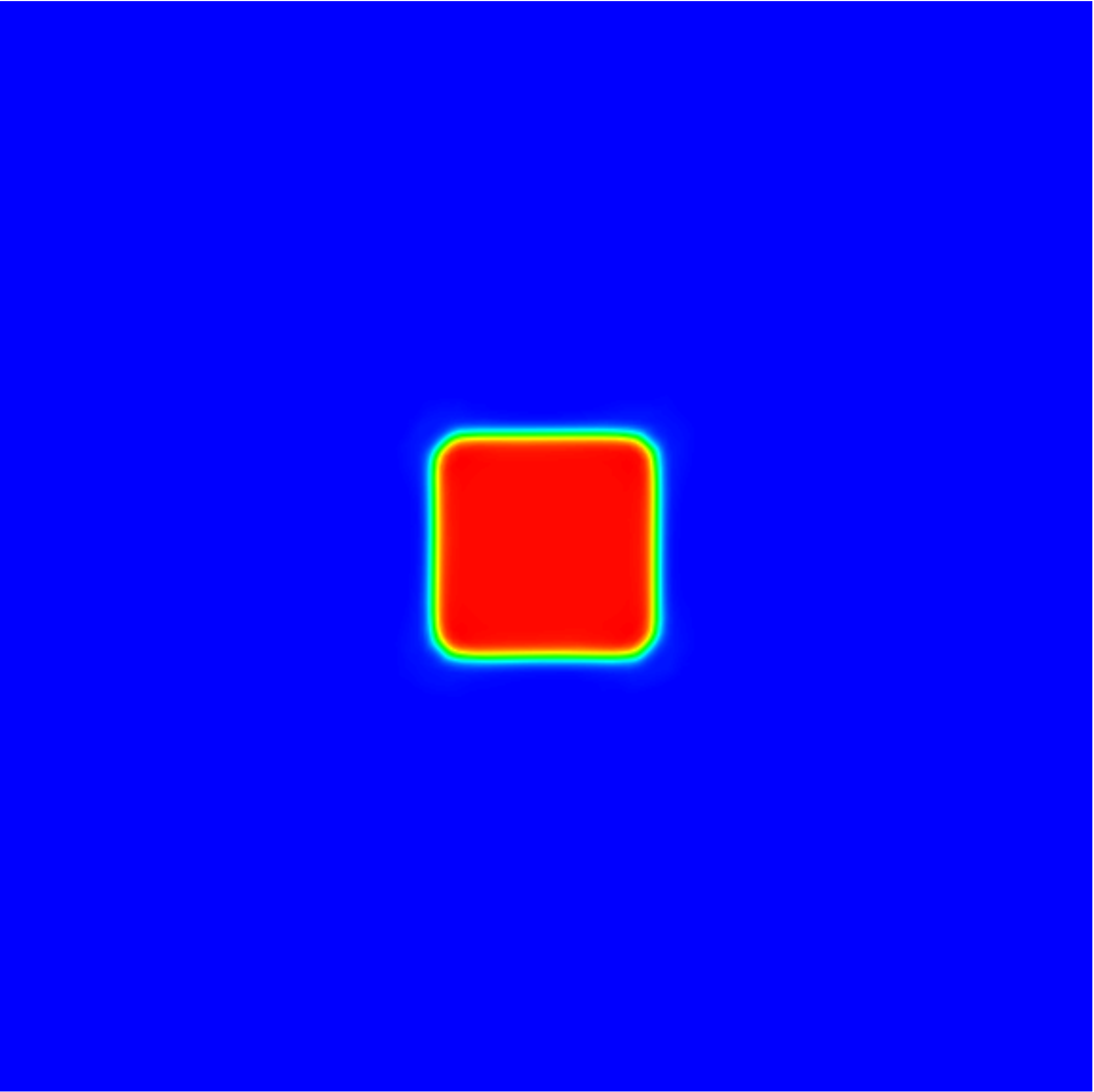}& \includegraphics[width=0.3\linewidth]{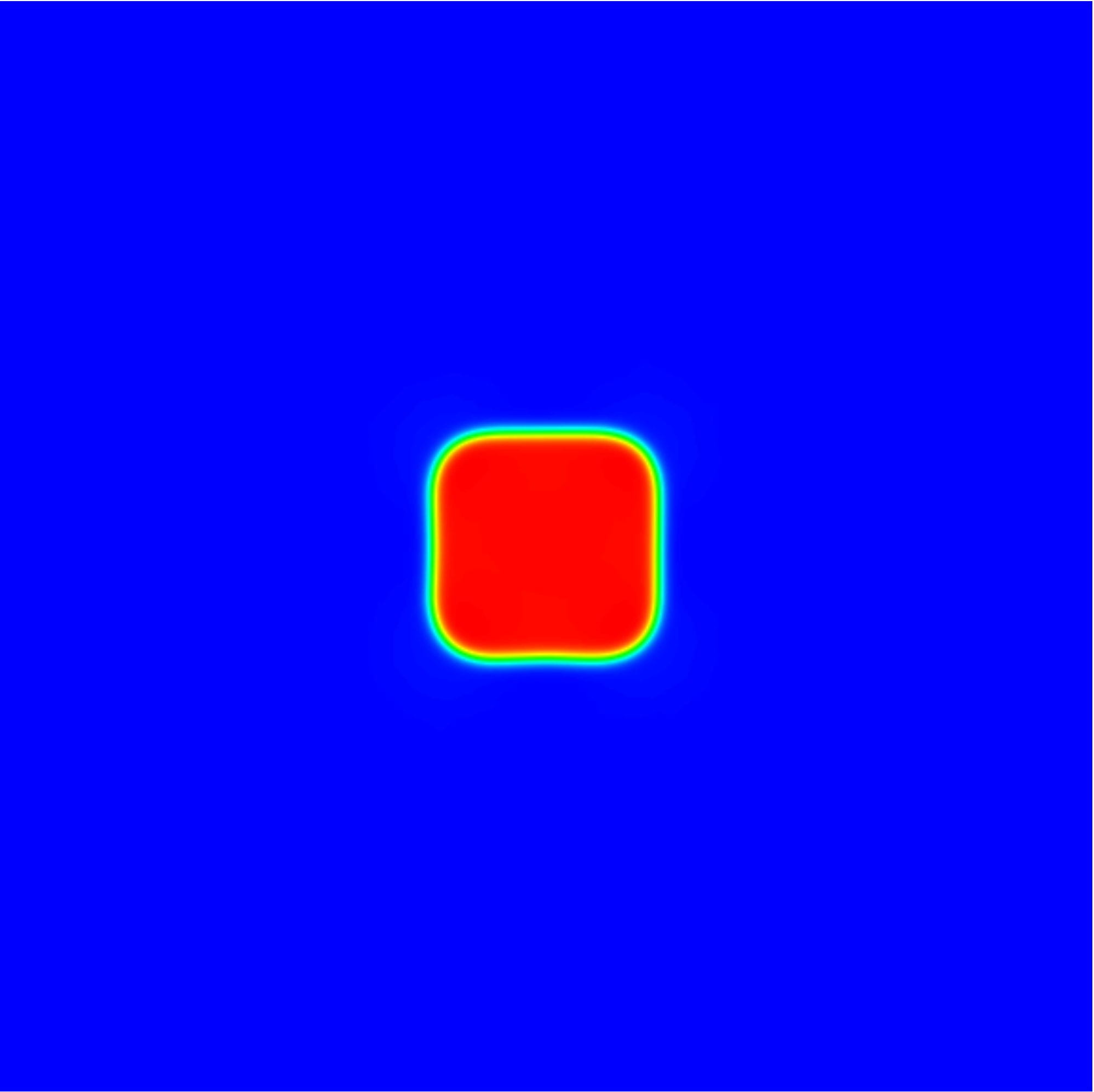}\\
   $t=0.02$ & $t=0.1$  \\
   \includegraphics[width=0.3\linewidth]{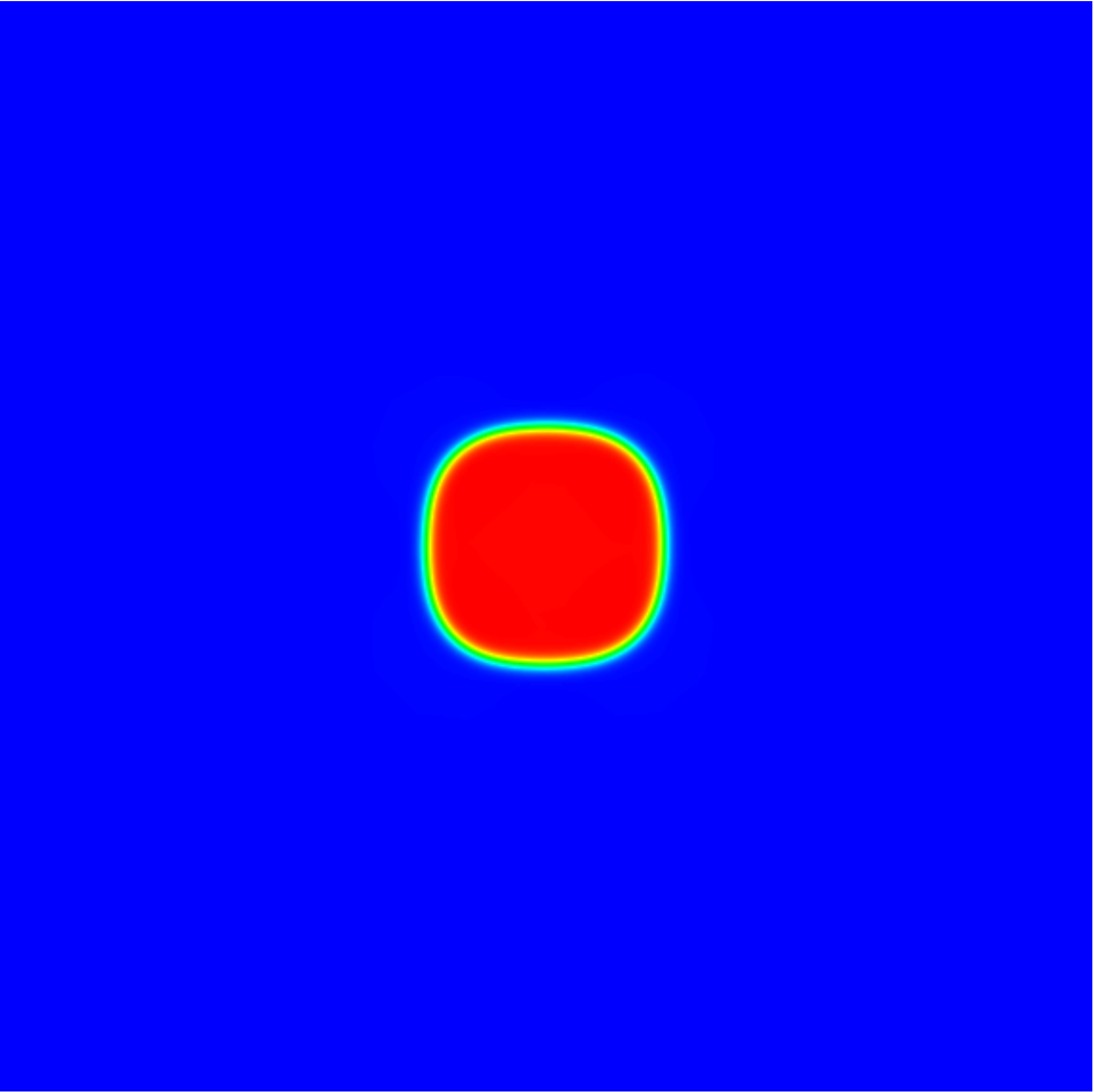}& \includegraphics[width=0.3\linewidth]{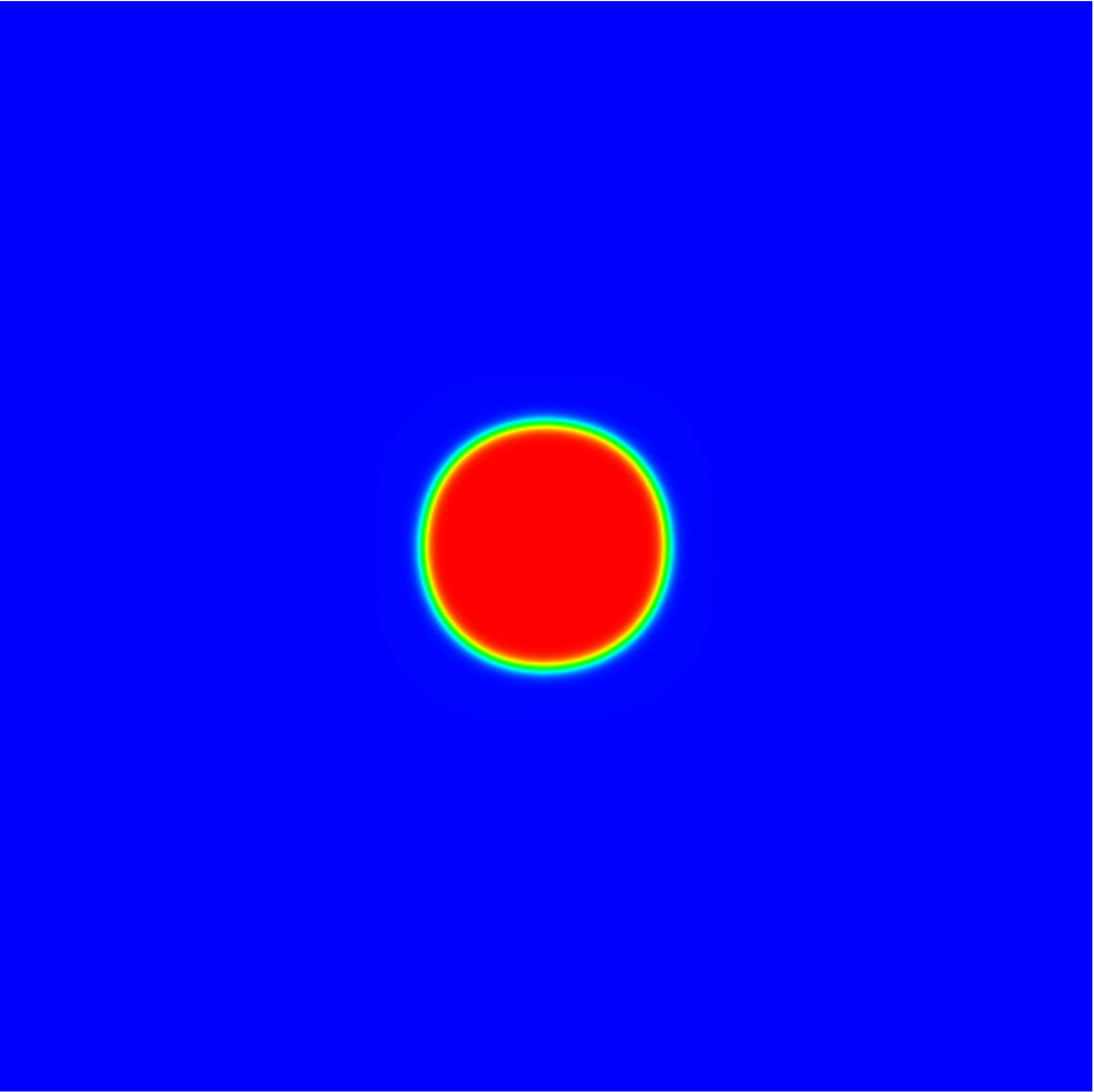}\\
   $t=0.4$ & $t=1$  
\end{tabular}
\caption{Shape relaxation of surface tension driven flow; $\epsilon=0.005$, $\sw=200$, $M(\phi)=0.1\sqrt{(1-\phi^2)^2+\epsilon^2}$, $Re=10$, $\delta t=0.005$; Adaptive mesh refinement is explored for spatial discretization.}
\label{ShaRE}
\end{figure}

\begin{figure}[h!]
\centering
\begin{tabular}{cc}
  \includegraphics[width=0.3\textwidth]{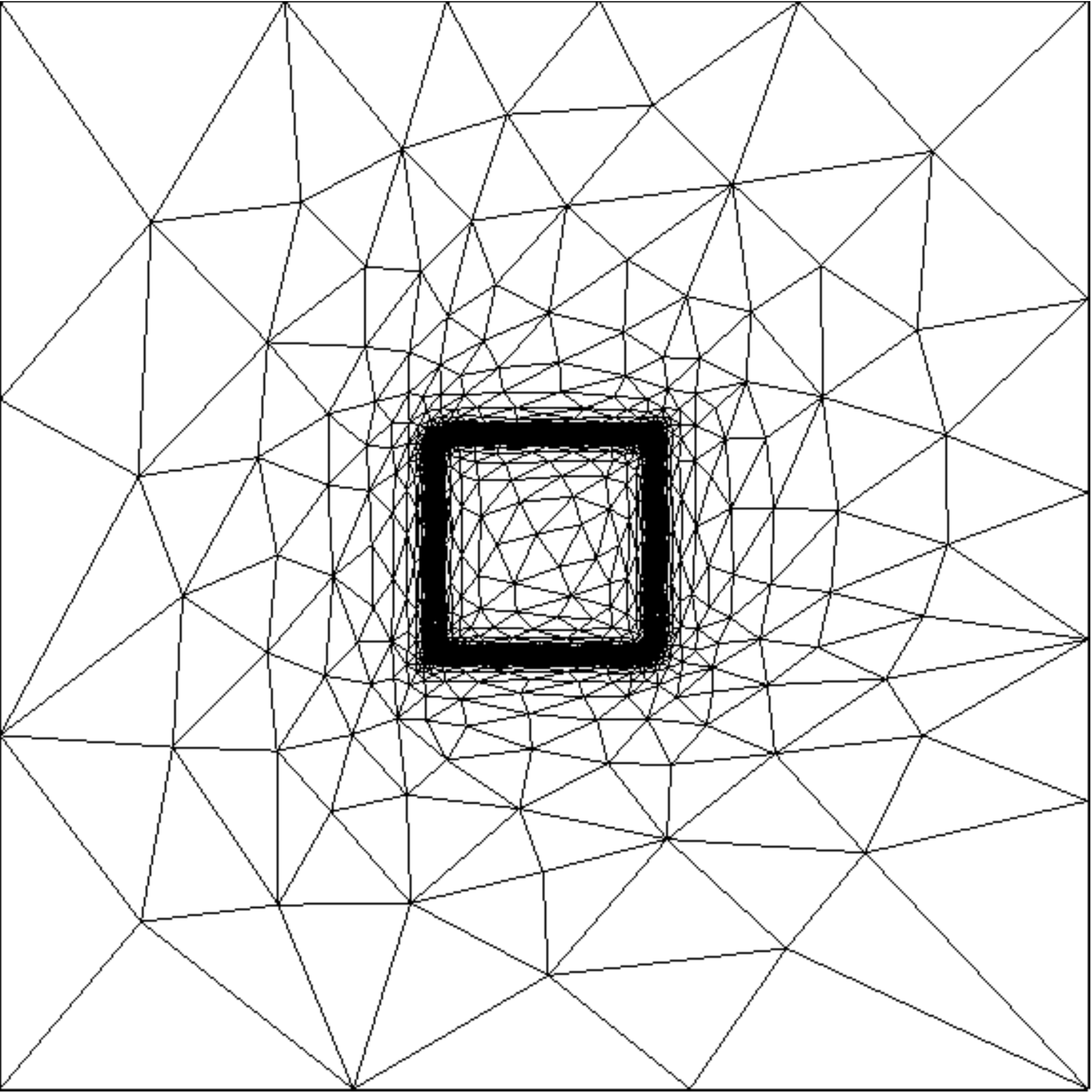}& \includegraphics[width=0.3\textwidth]{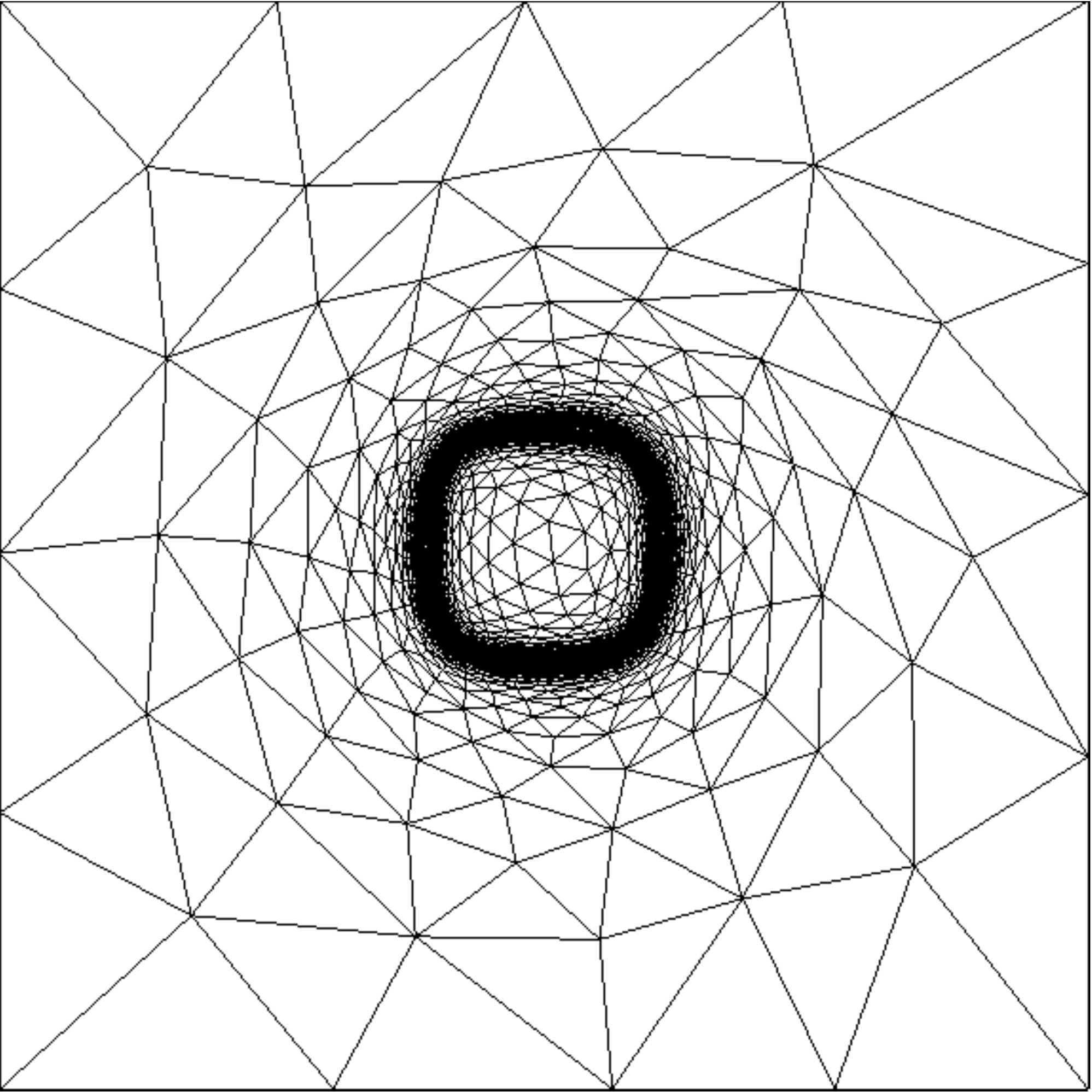}\\
   mesh at $t=0.02$ & mesh at $t=0.4$   
\end{tabular}
\caption{Adaptive mesh refinement associated with shape relaxation in Fig. \ref{ShaRE} at $t=0.02, 0.4$; $\epsilon=0.005$, $4$ grid elements are placed across the interfacial area.}
\label{AdpM}
\end{figure}


Next, we demonstrate the effect of imposed shear on shape relaxation. The initial configuration of order parameter is  given in Figure \ref{IniSh}. For velocity,  we take the initial data to be the Stokes solution to the lid driven cavity problem and for boundary data we take $\bfu|_{y=1}=\big(x(1-x), 0\big)$ and zero otherwise. We set $Re=100$ and the rest of the parameters are the same as in the case of surface tension driven flow (Fig. \ref{ShaRE}). The relaxation of the shape under shear driven flow  and the associated flow field are reported  in Fig. \ref{ShearD}. As the flow goes clockwise, the shape travels slightly to the left. Meanwhile,  the shape elongates to an ellipse with the major axis along north-west direction.

\begin{figure}[h!]
\centering
\begin{tabular}{ccc}
  $t=0.02$ & \includegraphics[width=0.26\linewidth]{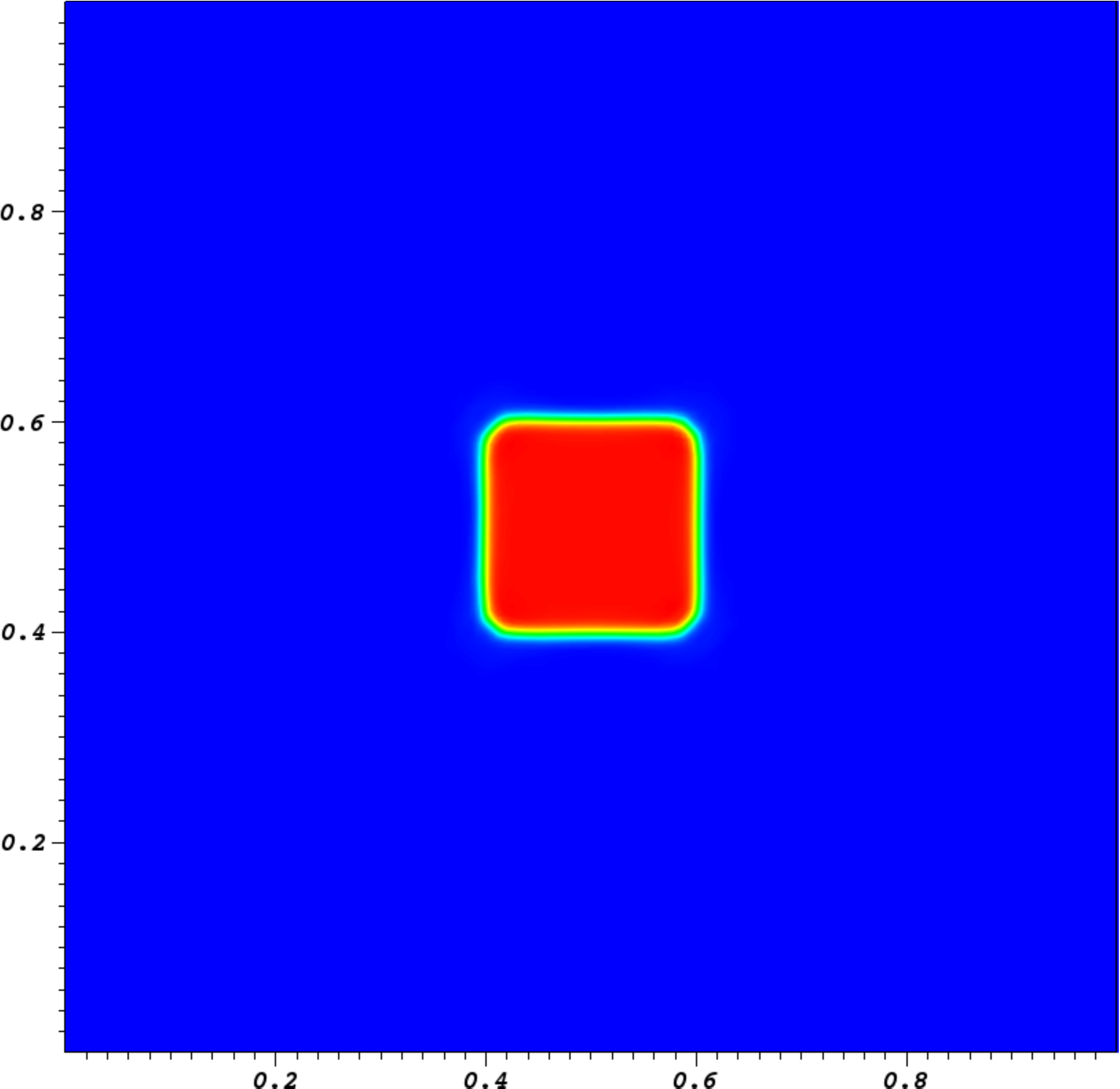}&  \includegraphics[width=0.29\linewidth]{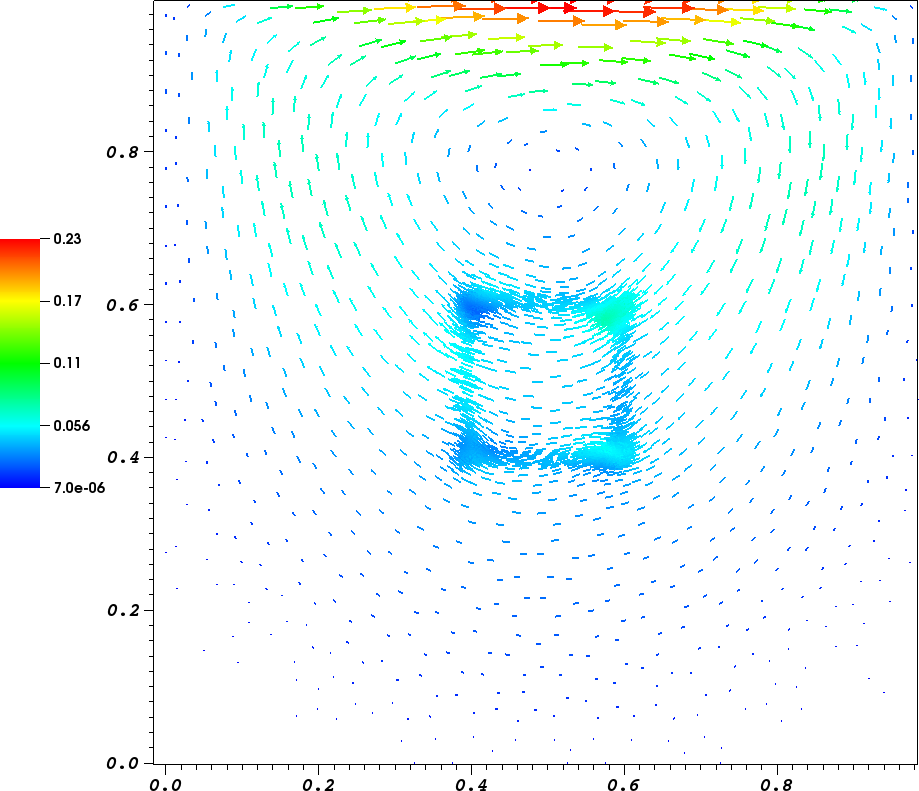}\\
  $t=0.1$& \includegraphics[width=0.26\linewidth]{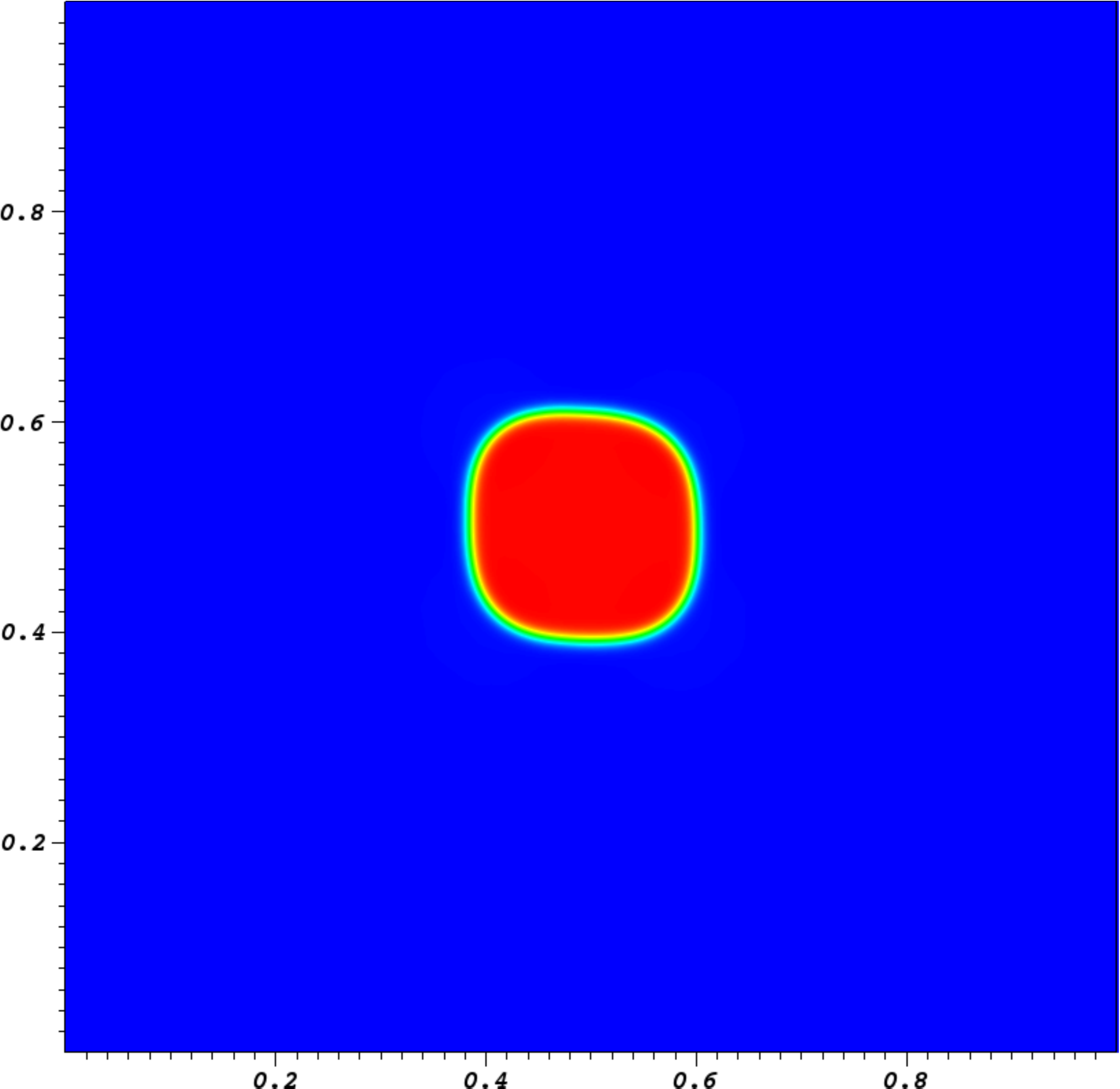} &\includegraphics[width=0.29\linewidth]{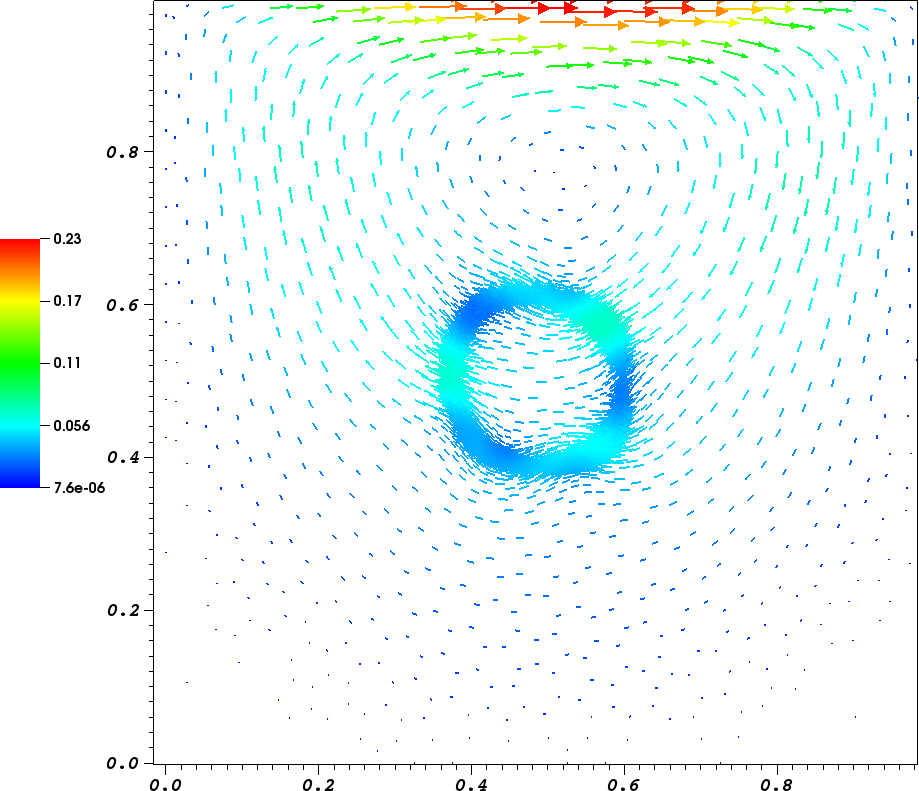}  \\
   $t=0.4$ & \includegraphics[width=0.26\linewidth]{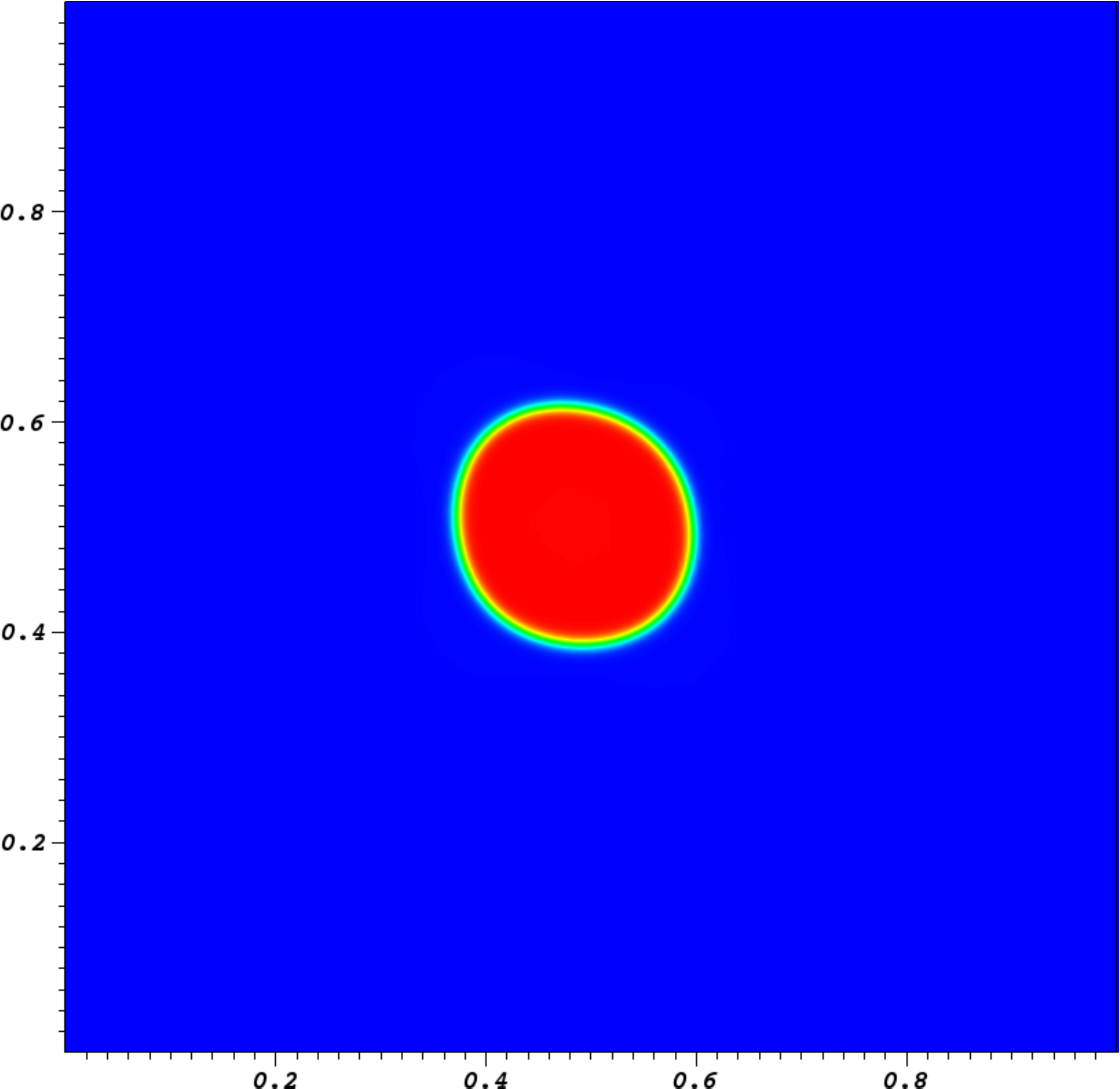}&\includegraphics[width=0.29\linewidth]{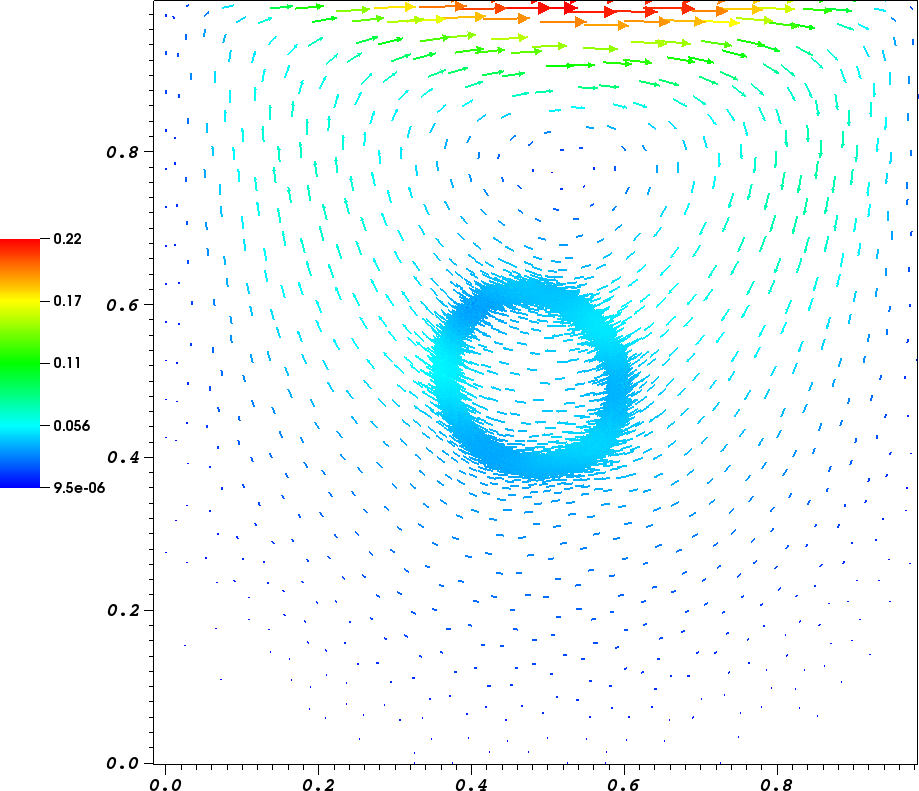} \\
   $t=1$  &\includegraphics[width=0.26\linewidth]{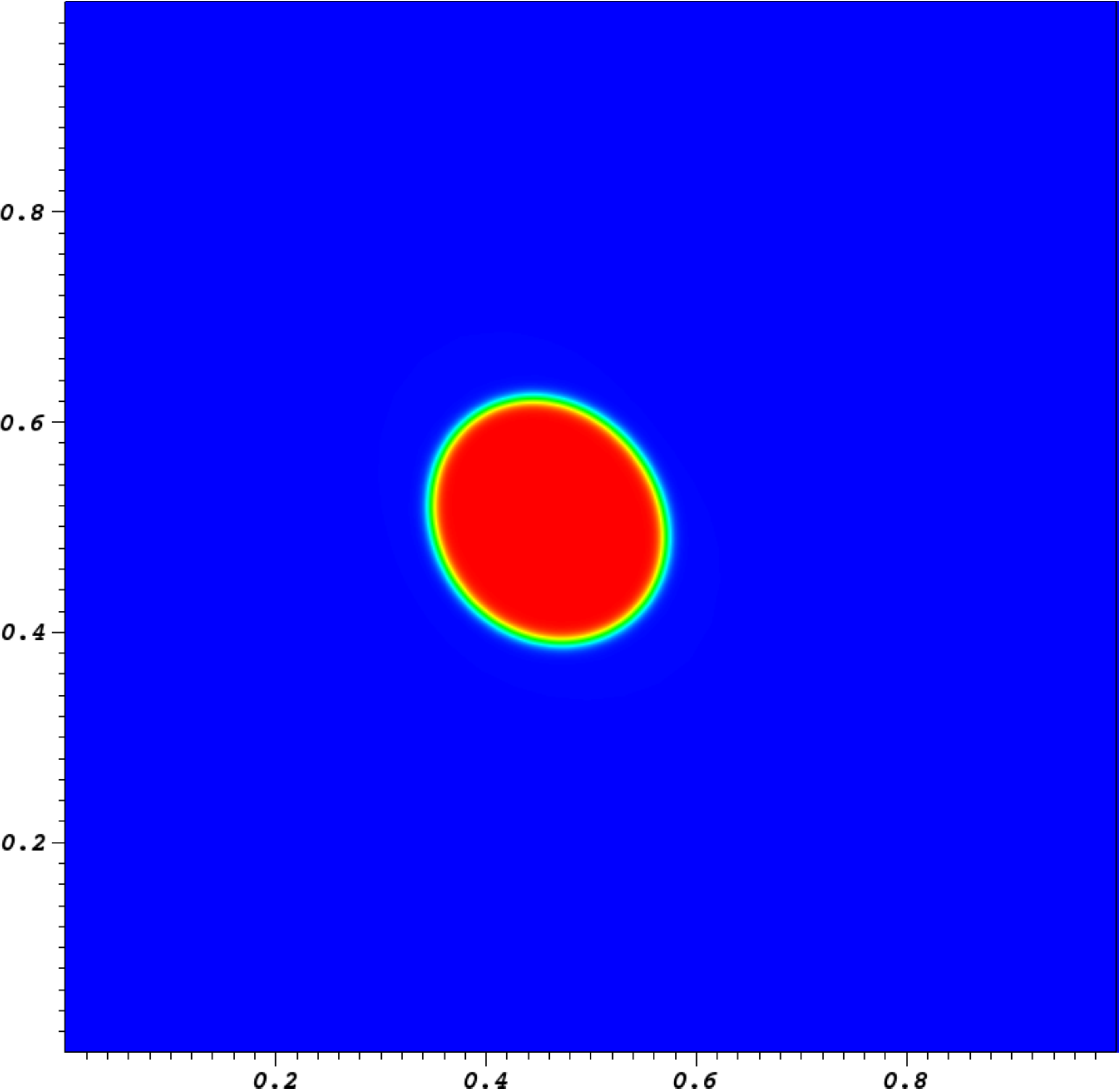} &\includegraphics[width=0.29\linewidth]{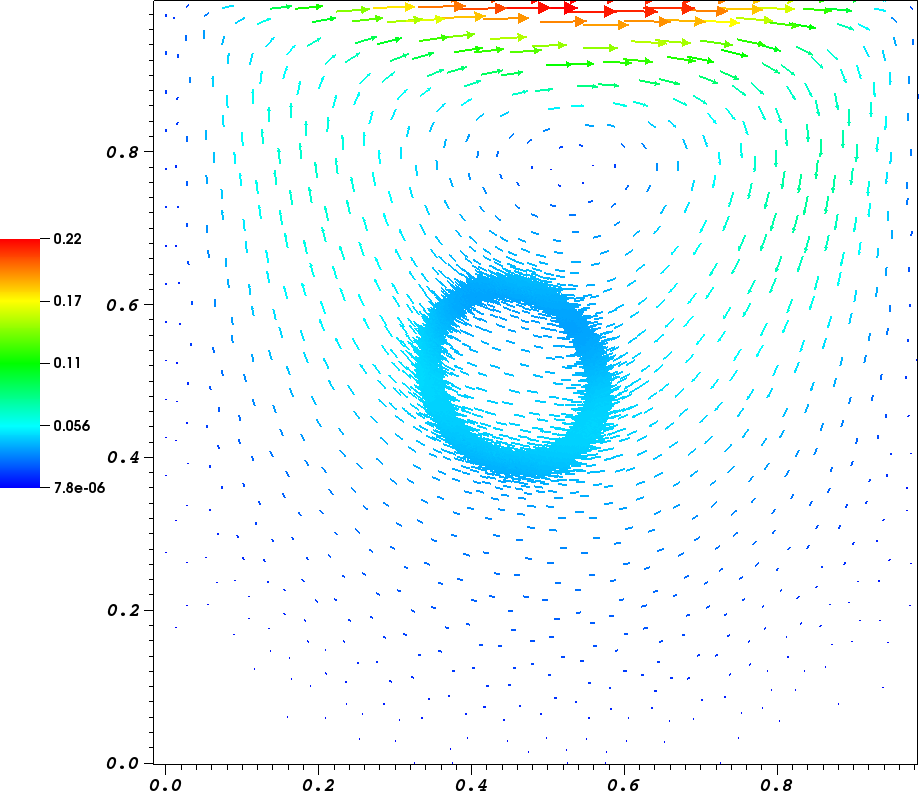}
\end{tabular}
\caption{Shape relaxation under shear driven flow and the flow field; The applied shear is on the upper boundary with a shear rate of $x(1-x)$; $Re=100$, $\epsilon=0.005$, $\sw=200$, $M(\phi)=0.1\sqrt{(1-\phi^2)^2+\epsilon^2}$.}
\label{ShearD}
\end{figure}

 



\subsection{Spinodal decomposition}
The CHNS system \eqref{CH}-\eqref{div} can be used as a model for spinodal decomposition of a binary fluid, cf. \cite{KKL2004}. Here we examine the effect of the excess surface tension, defined as $\gamma=\frac{1}{\sw}$,  on coarsening during spinodal decomposition. The initial velocity is zero $\bfu_0=0$. For the initial condition of the phase field variable, we take a random field of values $\phi_0=\bar{\phi}+r(x,y)$ with an average composition $\bar{\phi}=-0.05$ and random $r \in [-0.05, 0.05]$. We take no-slip no penetration boundary condition for velocity and homogeneous Neumann boundary condition for $\phi$ and $\mu$. The parameters are $\epsilon=0.005$, $M(\phi)=0.1\sqrt{(1-\phi^2)^2+\epsilon^2}$, $Re=10$, $\delta t=0.005$, $h=\frac{\sqrt{2}}{256}$. In Fig. \ref{Spi}, we show some snapshots of the filled contour of $\phi$ in gray scale (white color $\phi\approx 1.0$, black color $\phi \approx -1.0$) at different times with $\gamma=0,  0.1\epsilon, 1.0 \epsilon$, respectively. The case of $\gamma=0$,  corresponds to purely Cahn-Hilliard equation with no fluid motion, is included for comparison purpose. 

\begin{figure}[h!]
\centering
\begin{tabular}{cccc}

$t=4$ & \includegraphics[scale=0.168]{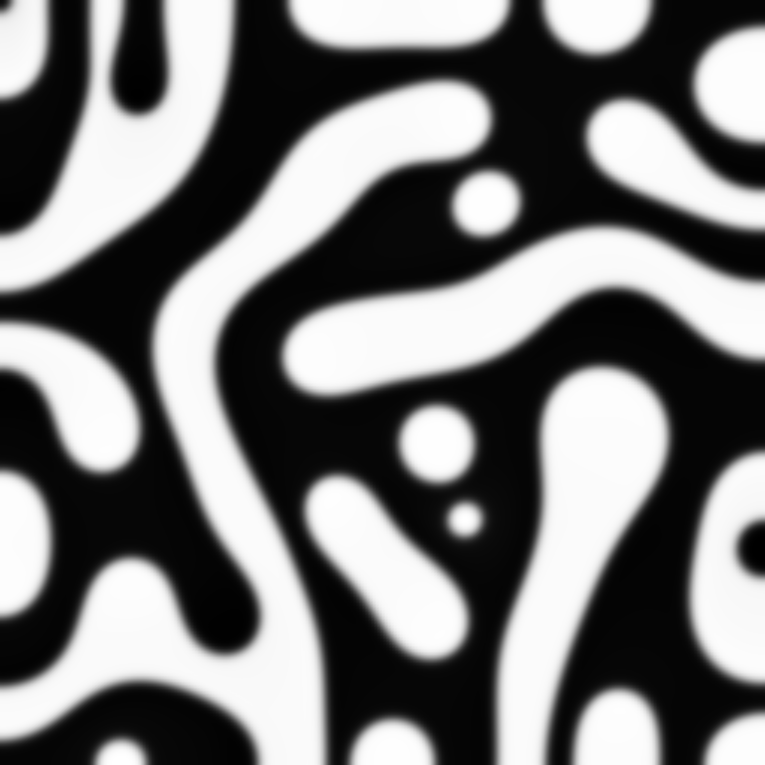}& \includegraphics[scale=0.18]{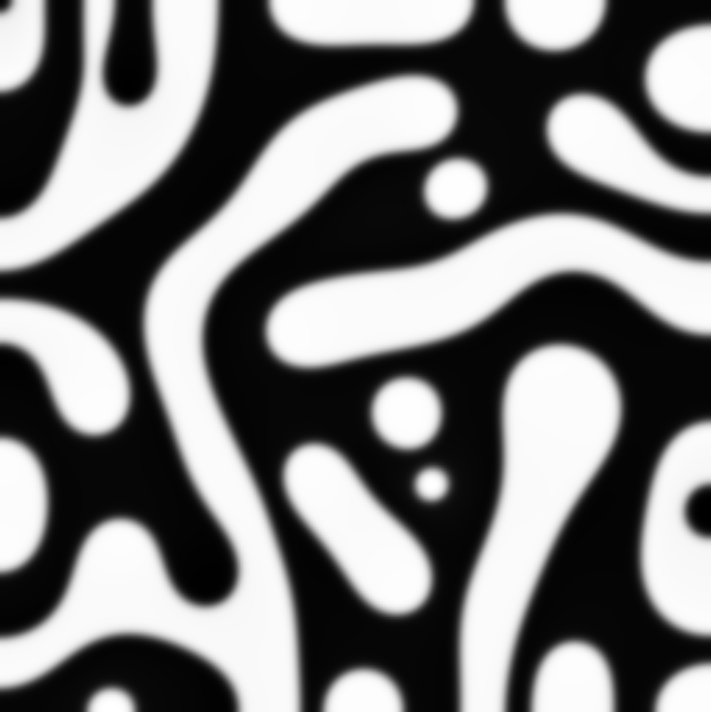} &\includegraphics[scale=0.18]{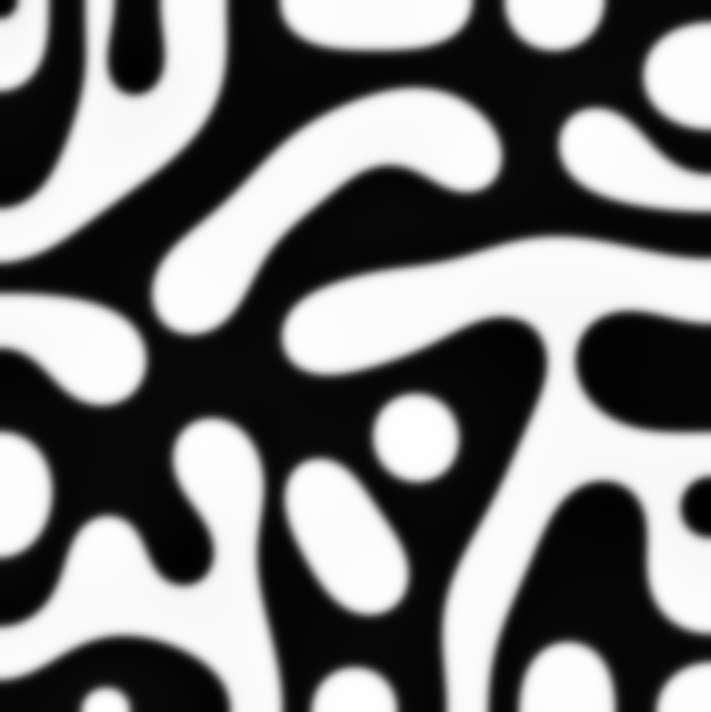} \\
 
$t=9$ & \includegraphics[scale=0.168]{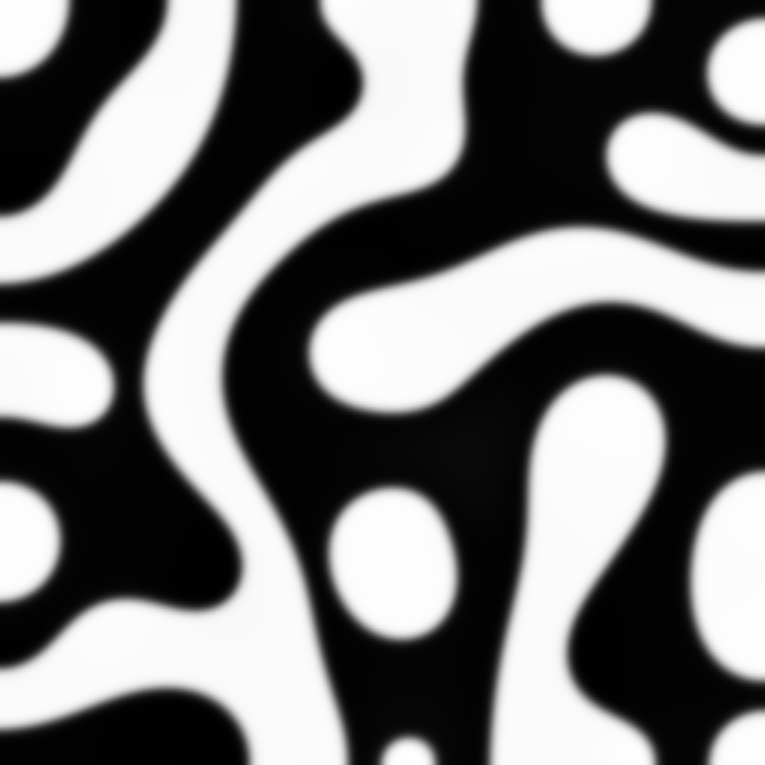}& \includegraphics[scale=0.18]{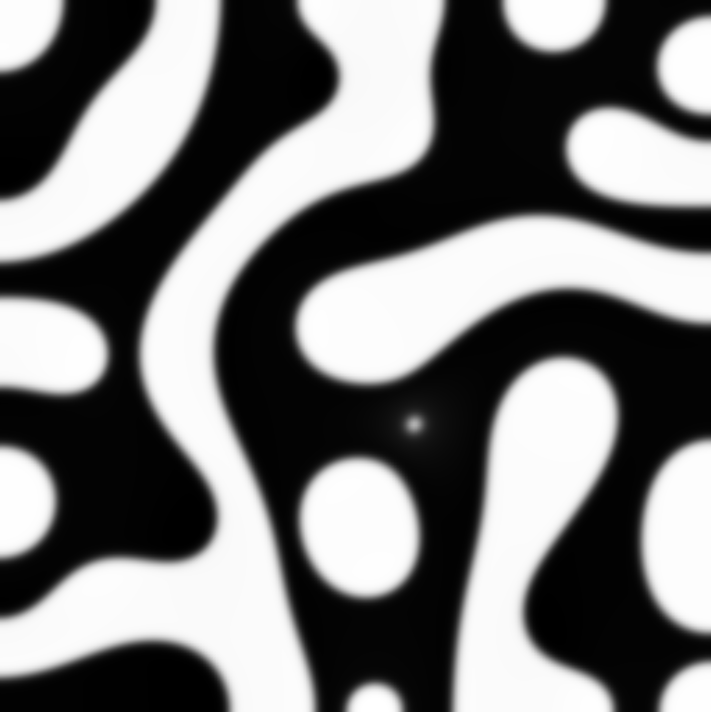} &\includegraphics[scale=0.18]{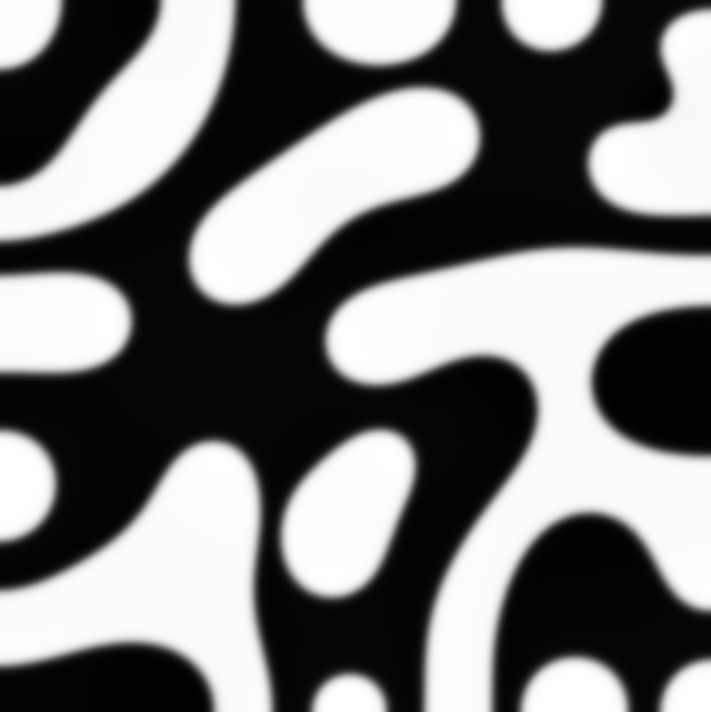}\\
 
 $t=20$ & \includegraphics[scale=0.168]{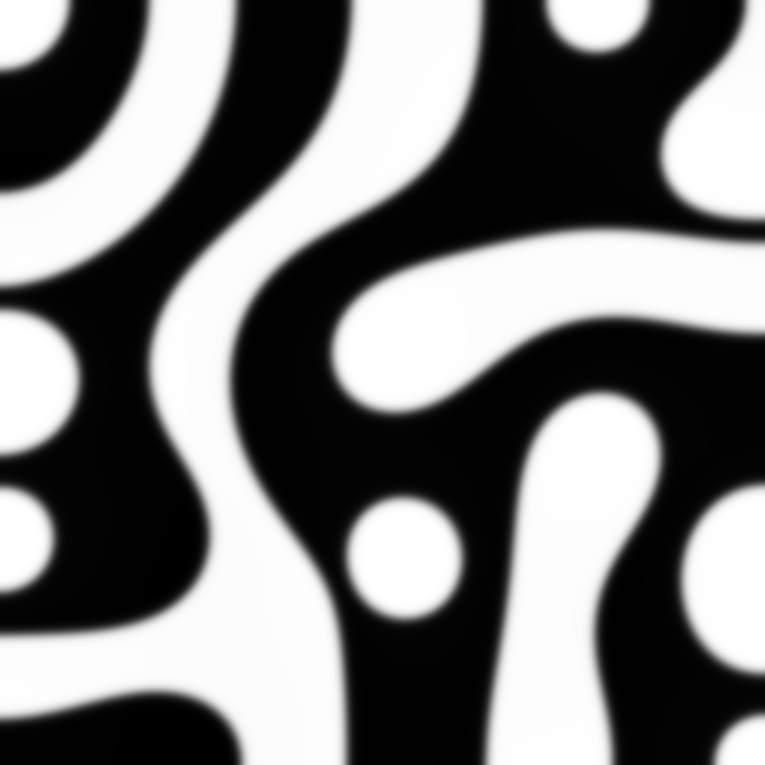}& \includegraphics[scale=0.19]{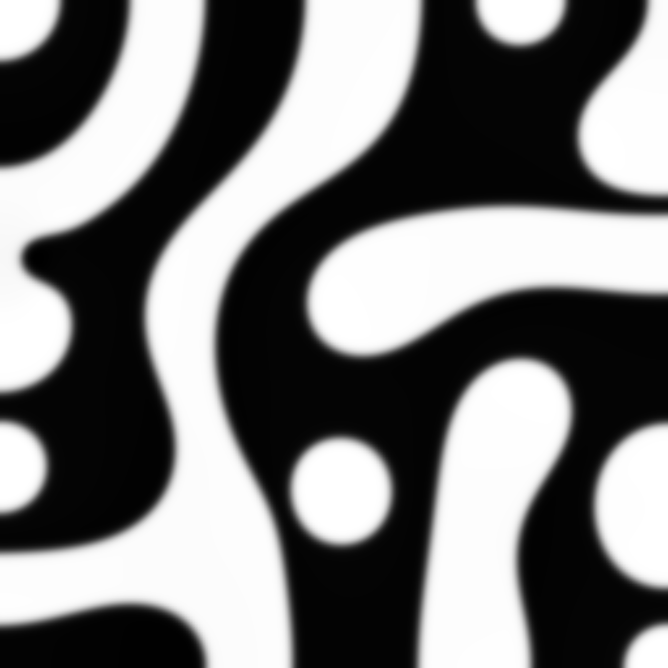} &\includegraphics[scale=0.19]{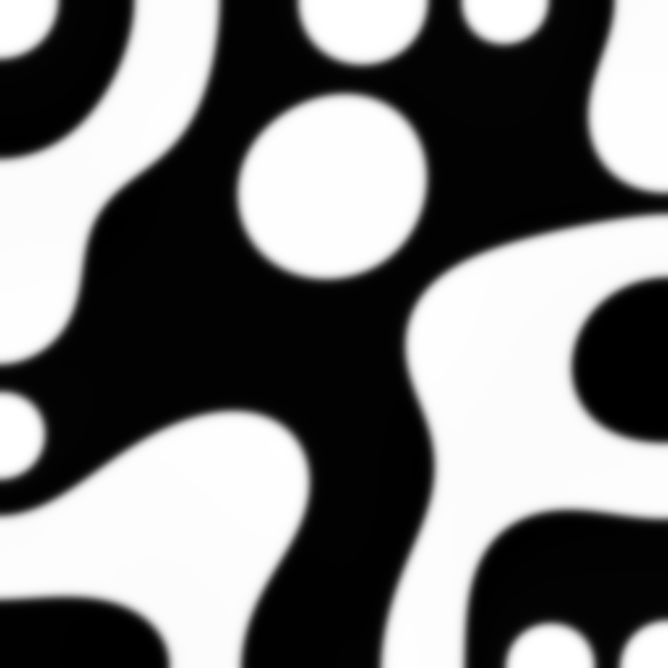}\\
 
$t=40$ & \includegraphics[scale=0.168]{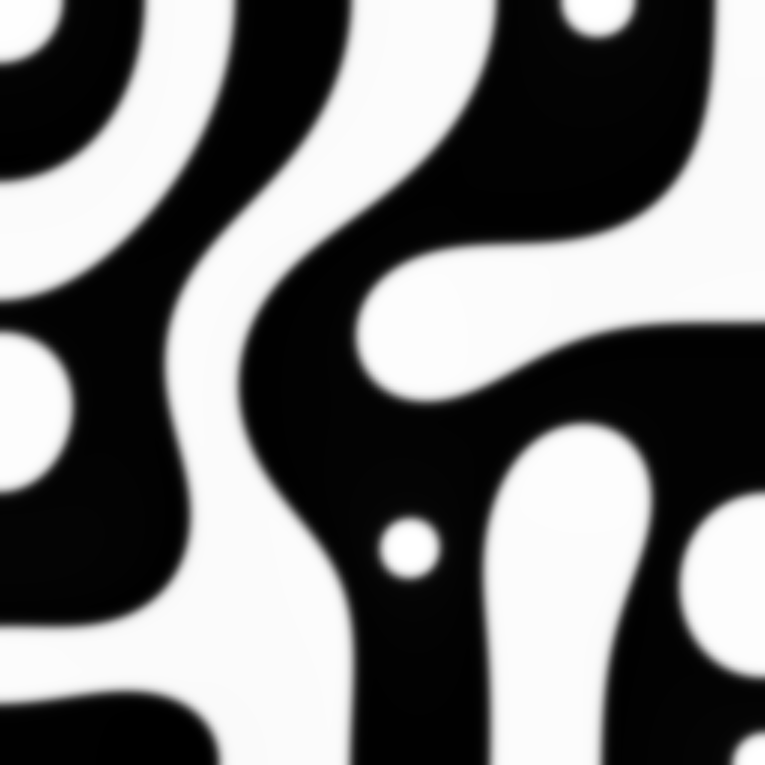}& \includegraphics[scale=0.19]{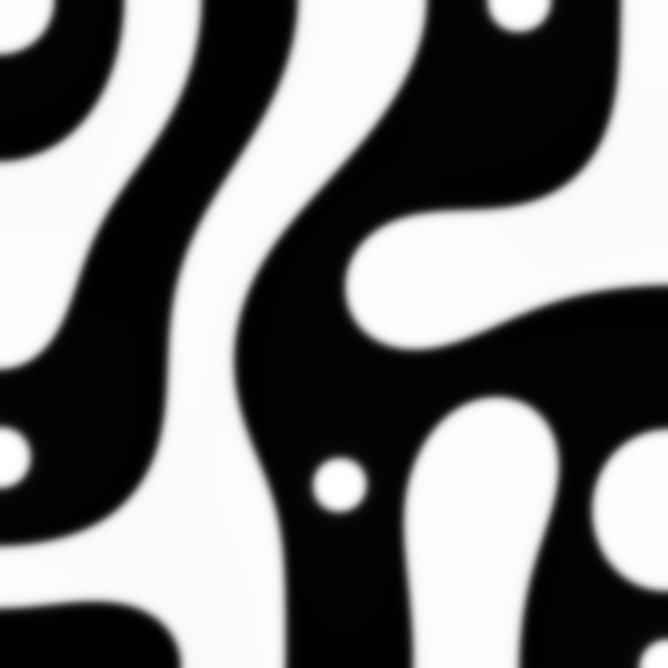} &\includegraphics[scale=0.19]{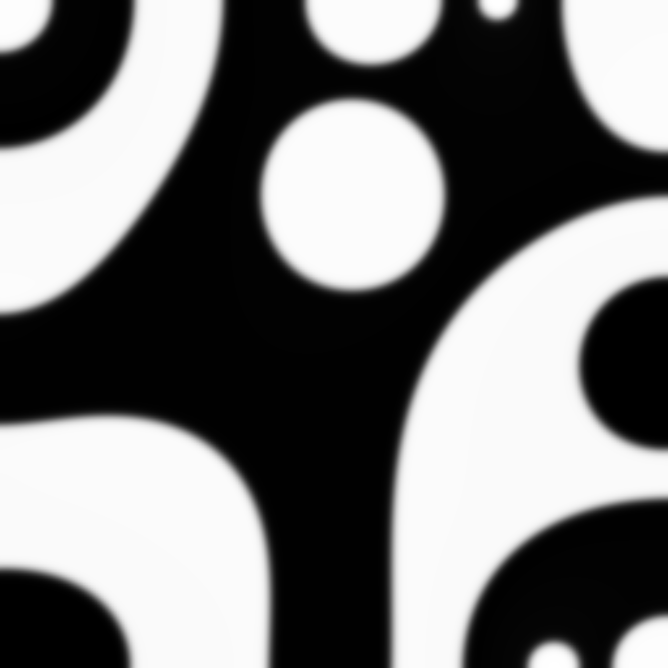}\\
   & $\gamma=0$& $\gamma=0.1 \epsilon$ &$\gamma=1.0 \epsilon$   
 
\end{tabular}
\caption{Snapshots of coarsening of a binary fluid during spinodal decomposition with $\gamma = 0.1\epsilon$ (second column), $1.0 \epsilon$ (third column), respectively; The case of $\gamma=0$ (first column) is included for comparison purpose; The rest of the parameters are $\epsilon=0.005$, $M(\phi)=0.1\sqrt{(1-\phi^2)^2+\epsilon^2}$, $Re=10$, $\delta t=0.005$, $h=\frac{\sqrt{2}}{256}$.}
\label{Spi}
\end{figure}

After a rapid initial phase separation (not shown in Fig. \ref{Spi}), the dynamics of the CHNS system are dominated by the slow process of coarsening. There are several physical mechanisms in the CHNS system that contribute to the coarsening process: bulk diffusion, surface diffusion, and hydrodynamic convection.  Note that our chosen regularized degenerate mobility $M(\phi)$ limits the bulk diffusion (of order $\epsilon$) at the late stage of the coarsening process. In comparison to the coarsening process governed by the Cahn-Hilliard equation with no fluid motion (the first column of Fig. \ref{Spi}), we find that the hydrodynamic effect speeds up the coarsening process by promoting the droplets coalescence, the larger $\gamma$, the more dramatic the coalescence effect. The effect is less discernible in the case of $\gamma=0.1 \epsilon$ ($\sw=2000$). Indeed, the morphology for $\gamma=0$ ($\sw=\infty$) and $\gamma=0.1\epsilon$  are nearly identical over the evolution. One can even observe the evaporation-condensation effect (Ostwald ripening) for the scattered isolated drops at $t=4, 9$. In contrast, for $\gamma=\epsilon$ ($\sw=200$) at the same time, the morphology is less deformed and exhibits rich connection (fewer isolated drops). Moreover, as time evolves, the scattered islands quickly merge together.    

Coarsening rate can be tied with surface energy decay rate.
The domain size of one phase $L$ (physical length scale)  can be defined as a suitable negative norm of the order parameter \cite{KoOt2002}. Recall that the surface energy $E_f$ is defined as
$$
E_f=\frac{1}{\sw}\int_{\Omega} \big( \frac{1}{\epsilon} f_0(\phi)+\frac{\epsilon}{2}|\nabla \phi|^2\big)\, dx.
$$
Thus the surface energy $E_f$ is proportional to the average interfacial circumference in 2D (interfacial area in 3D), at least near equilibrium where the order parameter roughly has a profile of hyperbolic tangent function \cite{Shen2012}. It follows from the conservation of volume that the spatially averaged surface energy should scale like the inverse of the domain diameter $L$.  This heuristic argument suggests that the decay rate of the surface energy can be used as a proxy of the phase coarsening rate. One can also motivate this argument from the standpoint of sharp interface limit. The exact relation between $L$ and $E_f$ is an inequality established rigorously  in \cite{KoOt2002}.  Fig. \ref{CoarsenRate} shows the correlation between the surface energy decay rate $E_f$ (thus coarsening rate) and the excess surface tension parameter $\gamma$. It is observed that larger surface tension $\gamma$ (smaller $\sw$) yields faster coarsening rate, which agrees with the energy law \eqref{ModEnergyLaw}.


\begin{figure}[h!]
\centering
 \includegraphics[width=0.7\linewidth]{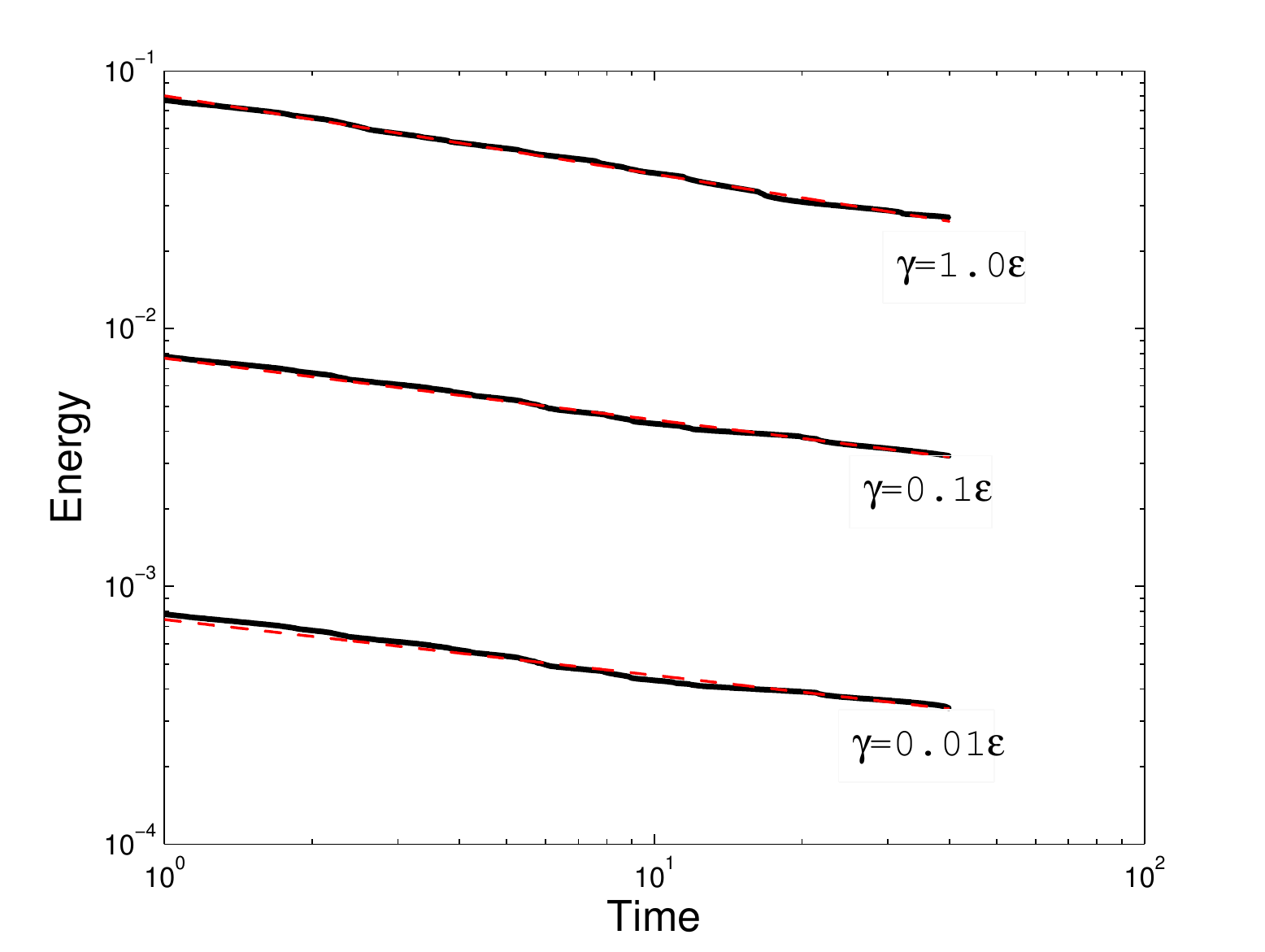}
 \caption{Loglog plot of the surface energy $E_f$ as a function of time (solid lines) for simulations in Fig. \ref{Spi}. Here we include the case $\gamma=0.01\epsilon$ for comparison purpose. The dash lines are fitted functions $c_1 t^{-0.216}$ ($\gamma=0.01\epsilon$),  $c_2 t^{-0.239}$ ($\gamma=0.1\epsilon$) and $c_3t^{-0.304}$ ($\gamma=1.0\epsilon$), respectively.}
  \label{CoarsenRate}
\end{figure}

For a large system of a binary fluid at late stage of spinodal decomposition, it is expected \cite{Siggia1979, MGG1985} that the coarsening rate  would obey a dynamical scaling law: $L(t) \propto t ^\alpha$, where $L(t)$ is the average domain size of one phase. Nevertheless, in 2D such a scaling law is open to debate (see the recent work \cite{OSS2013} and references therein). Here we run our scheme on a domain $\Omega=[0,200] \times [0,200]$ for a final time up to $10^4$. The parameters are: $\epsilon=1.0$, $M=1.0$, $\sw=1.0$, $Re=1.0$. The initial and boundary conditions for $\phi$ and $\bfu$ are set similarly as above. We take $\delta t =0.5$ and $h=\sqrt{2}$ for $t \leq 1000$, and $\delta =1.0$ and adaptive mesh refinement for $t \in [10^3, 10^4]$. We plot the decay of the surface energy $E_f$ in log-log scale in Fig. \ref{Coarsening}, which reveals roughly a decay rate of $\frac{1}{2}$ at the late stage of coarsening. This result corroborates the $t^{\frac{1}{2}}$ growth law for the average domain size proposed in \cite{MGG1985}. Note that the wall effect becomes influential when $t$ approaches $10^4$ at which large islands occupy the boundary of the domain.

\begin{figure}[h!]
\centering
 \includegraphics[width=0.7\linewidth]{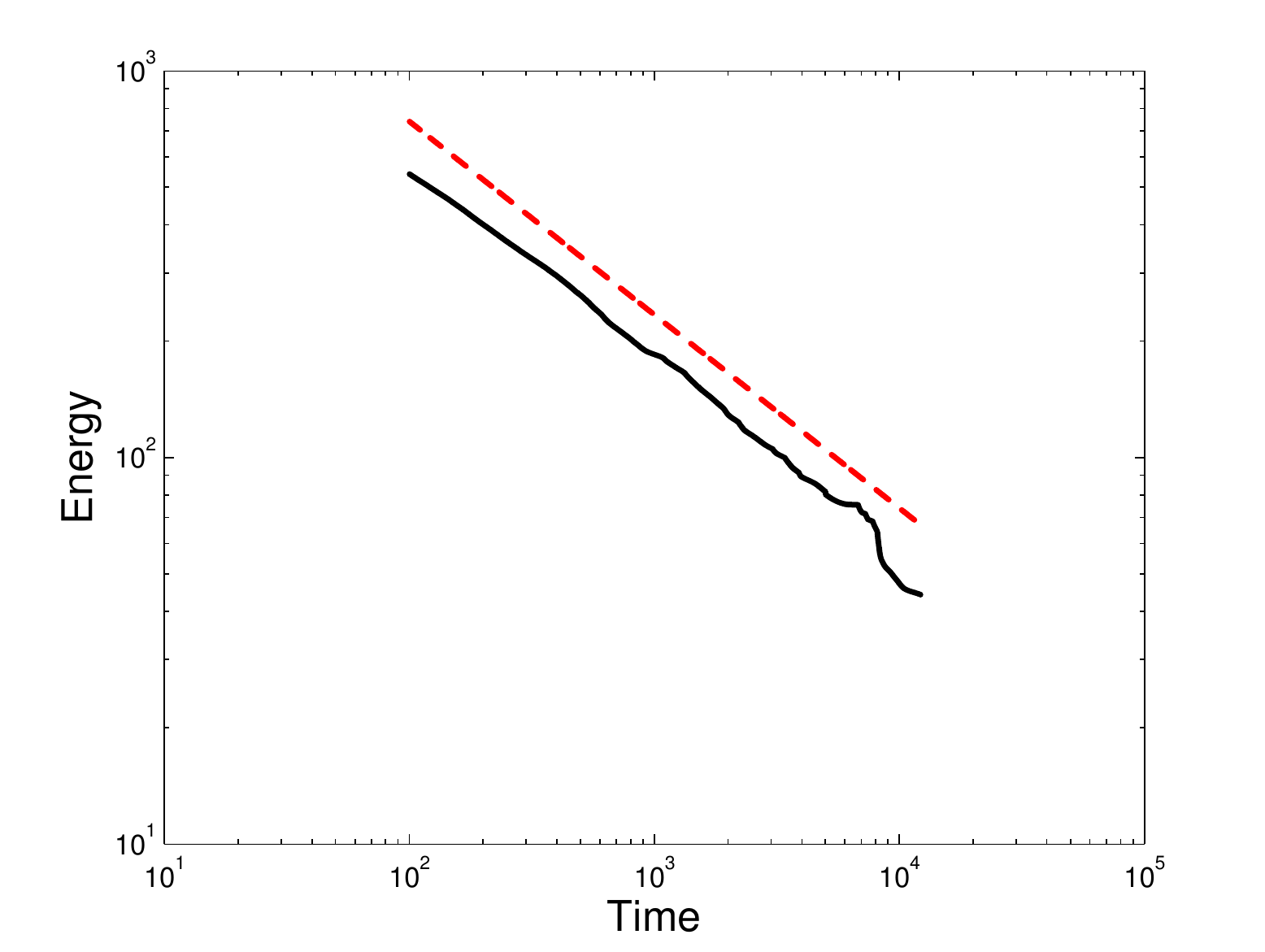}
 \caption{Loglog plot of the surface energy $E_f$ as a function of time (solid line) for the CHNS system; $\Omega=[0, 200] \times [0, 200]$, $\epsilon=1.0$, $M=1.0$, $\sw=1.0$, $Re=1.0$; The red dash line has a slope of $-\frac{1}{2}$. }
  \label{Coarsening}
\end{figure}

\section{Conclusions}

In this paper, we have presented a novel second order in time numerical method for the Cahn-Hilliard-Navier-Stokes system that models two-phase flow with matched density.
The method is efficient since we decoupled the pressure from the velocity and phase field, and the coupling between the velocity field and the phase field is weak.
We have shown in a rigorous fashion that the scheme is unconditionally stable and uniquely solvable. Fully discrete numerical methods effected with finite-element method are also presented and analyzed with similar conclusions. To the best of our knowledge, this is the first second-order scheme that decouples the pressure and the velocity and phase field variables while maintaining unconditional stability and unique solvability.

Several numerical experiments are performed to test the accuracy of the scheme. We verify numerically that our scheme is conservative, energy- dissipative, and is of second order accuracy in $L^2$ norm. We demonstrate the effectiveness of our scheme incorporated with adaptive mesh refinement by simulating the shape relaxation with and without applied shear. Finally, we also investigates the effect of surface tension on the coarsening rate of spinodal decomposition of a binary fluid. In particular, our long time numerical simulation suggests a growth rate of $t^{\frac{1}{2}}$ for a large system at late stage, which agrees with \cite{MGG1985}.

There are numerous potential extensions of the current work. The design of second-order in time scheme that decouples the pressure, velocity and phase field completely, and  is  unconditionally stable and uniquely solvable is very desirable. The extension of the current scheme to the case of unmatched density, or to the case of coupled Cahn-Hilliard-Stokes-Darcy system that models two-phase flow in karstic geometry would also be interesting \cite{HSW13}.
From the theoretical side, the rigorous error analysis of the scheme, especially with adaptive mesh,  is a very attractive but challenging topic.  
  
\section*{Acknowledgments} 
This work was completed while Han was supported as a Research Assistant on an NSF grant (DMS1312701). The authors also acknowledge the support of NSF DMS1008852, a planning grant and a multidisciplinary support grant from the Florida State University.

\bibliographystyle{plain}
\bibliography{multiphase.bib}

\end{document}